\documentclass[10pt, oneside]{article}

\input math_headers.sty

\renewcommand{\d}{\mathrm{d}}
\newcommand{\sect}{\mathrm{sect}}
\newcommand{\cl}{\mathrm{cl}}
\newcommand{\Max}{\mathrm{Max}}
\newcommand{\pmp}{\Phi_p^{\mathrm{Max}}}

\begin{document}

\title{Abelian Duality for Generalized Maxwell Theories}
\author{Chris Elliott \\ {\small Institut des Hautes \'Etudes Scientifiques} \\ {\small 35 Route de Chartres} \\ {\small 91440, France}}
\date{}

\maketitle

\begin{abstract}
We describe a construction of generalized Maxwell theories -- higher analogues of abelian gauge theories -- in the factorization algebra formalism of Costello and Gwilliam, allowing for analysis of the structure of local observables.  We describe the phenomenon of abelian duality for local observables in these theories as a form of Fourier duality, relating observables in theories with dual abelian gauge groups and inverted coupling constants in a way compatible with the local structure.  We give a description of expectation values in this theory and prove that duality preserves expectation values.  Duality is shown to, for instance, interchange higher analogues of Wilson and 't Hooft operators.
\end{abstract}

\tableofcontents
\section{Introduction and Motivation}
The aim of this paper is to give a detailed account of the phenomenon of \emph{S-duality} in a very simple situation, as a duality between families of \emph{free} quantum field theories, in a way allowing explicit understanding and computation of the local structure of the duality.  This free version of S-duality was called \emph{abelian duality} by Witten in his 1995 paper on the subject \cite{WittenAbelian}, in which he addresses a relationship between the partition functions of abelian pure Yang-Mills theories on a compact 4-manifold.  This relationship, and generalizations to higher degree, were also studied by Verlinde in \cite{Verlinde}.  Analogues of this duality in lower dimensions were also known, for instance the \emph{T-duality} between sigma models with dual torus targets on a 2-manifold, and a duality between sigma models and abelian gauge theories on a 3-manifold (these are described in \cite{WittenQFS},  explicit calculations of duality for the partition functions in three dimensions have been performed by Prodanov and Sen \cite{ProdanovSen} and by Broda and Duniec \cite{BrodaDuniec}, and a detailed analysis of duality for more general observables was recently performed by Beasley \cite{Beasley1, Beasley2}).  These theories fit into an infinite family of theories whose fields model connections on higher torus bundles, models for which have been described by Freed \cite{Freed} using \emph{ordinary differential cochains} to model the fields.  The quantizations of these theories were further studied by Barb\'on \cite{Barbon}, who discussed abelian duality for the partition functions in higher degree theories, and Kelnhofer \cite{Kelnhofer}, who explained how to compute the vacuum expectation values of gauge invariant observables in these theories using the language of ordinary differential cochains.

Abelian duality for generalized Maxwell theories refers to a relationship between two different ``dual'' quantum field theories on a single spacetime manifold: one involving fields of degree $k$ and gauge group $T$ (higher gauge fields whose curvature is a $k+1$-form), and the other involving fields of degree $n-k$ and gauge group $T^\vee$.  The mathematical consequences of this relationship include an equality of the partition functions of the two theories (we do not address this here), but also the idea that there should exist a procedure that associates to a gauge invariant local observable a ``dual'' local observable in the dual theory, constructed by means of a Fourier transform on the space of fields.  For example, for ordinary abelian gauge theories in dimension 4, abelian duality should interchange Wilson and 't Hooft operators.  What's more, dual observables should have equal expectation values.  In this paper we will prove the following theorem, making this idea precise in the context of a specific mathematical model for local quantum observables. Fix a compact $n$-manifold $X$.

\begin{theorem}
Let $U \sub X$ be a fixed open subset of $U$.  To every gauge invariant local observable $\OO$ on the open set $U$ in a degree $k$ generalized Maxwell theory with gauge group $T$, there exists a (non-unique) gauge invariant \emph{dual} local observable $\wt \OO$ on $U$ in the theory of degree $n-k$ and with gauge group $T^\vee$.  If $\mr H^p(U;\ZZ) = 0$ then we can define vacuum expectation values of the two observables, and they agree:
\[\langle \OO \rangle_{T} = \langle \wt \OO \rangle_{T^\vee}.\]
\end{theorem}

Since the subset $U$ is not required to be compact, abelian duality makes sense \emph{locally}; that is, we can investigate in what sense duality is compatible with the inclusion maps from observables on an open subset $U \sub X$, and observables on a larger open set $V$ containing $U$.  So, we will describe abelian duality as a relationship between a pair of \emph{prefactorization algebras}, which model the local quantum observables in the quantum field theories.  The duality is compatible with the structure maps in the factorization algebras, but does not extend to a morphism of prefactorization algebras because of an obstruction to defining a dual observable on non-contractible open sets, where observables may not be determined purely by the curvature of a field.  Instead duality arises as a \emph{correspondence} of prefactorization algebras, by which we mean a zig-zag of morphisms of prefactorization algebras
\[\xymatrix{
&\obs_{\mr{intermediate}} \ar[dr] \ar[dl] & \\
\obs_{k,T} && \obs_{n-k, T^\vee}
}\]
relating two prefactorization algebras $\obs_{k,T}$ and $\obs_{n-k, T^\vee}$ on $X$ by means of an intermediate prefactorization algebra on the same space.

Despite the theories being free, duality of observables is still a non-trivial phenomenon to investigate.  The dual of a gauge-invariant observable can have a qualitatively different nature to the original observable.  For instance, in abelian Yang-Mills we verify that the dual of an abelian Wilson operator (a holonomy operator around a loop) is an 't Hooft operator (corresponding to imposing a singularity condition on the fields around a loop).

\begin{remark}
There are two aspects to the arguments of this paper.  First we explain how one can \emph{construct} a free quantum field theory -- a prefactorization algebra -- starting from a possibly non-linear space of fields.  We then demonstrate abelian duality for the theories so constructed starting from the fields and Lagrangian density of a generalized Maxwell theory.  These two steps are logically separate, in particular it should be possible to appreciate the duality results of Section \ref{sectionFourier} with only a cursory reading of the somewhat technical constructions in Sections \ref{sectionDerived} and \ref{sectionAction}.
\end{remark}

Now, we can outline the structure of the paper.  In Section \ref{sectionBV} we begin by describing the general formalism we use to construct the prefactorization algebra of quantum observables, starting from a sheaf of fields and an action functional.  This formalism (based on that of Batalin-Vilkovisky) was developed by Costello and Gwilliam \cite{Book2} \cite{GwilliamThesis} as a formulation of quantization techniques common in the physics literature (described for instance in Witten's expository note \cite{WittenAntibracket}) as homological algebra.  While we only use the theory for free theories here (their methods can also be used to study the perturbative parts of interacting theories), we do need to allow for spaces of fields that include discrete factors, so we go over the formalism with a certain amount of care.

Having set up the abstract theory we construct the main objects of study in Section \ref{sectionGenMaxwell}: the prefactorization algebras of observables in \emph{generalized Maxwell theories}.  These are free quantum field theories whose fields model connections on higher principal torus bundles, whose action functional generalizes the Yang-Mills action.  These theories are closely related to simpler free theories -- where the fields are just $p$-forms and the action is just the $L^2$ norm -- by mapping a connection to its curvature.  Observables of interest (such as Wilson and 't Hooft operators in abelian Yang-Mills) factor through the curvature, so in a sense ``come from'' this simpler theory.  As such we can prove results about the expectation values of observables purely in the world of curvatures.

The theory of expectation values arises naturally from the factorization algebra formalism.  In Section \ref{sectionExpVal} we describe how to abstractly define expectation values of gauge invariant observables by viewing the observables as living in a cochain complex with canonically trivialisable cohomology.  To compute expectation values we use classical physical techniques: \emph{Feynman diagrams} and \emph{regularization}.  Since the theory is free these methods are very well-behaved, and encode results about convergent sequences of finite-dimensional Gaussian integrals.

Finally, in Section \ref{sectionFourier} we introduce abelian duality for observables in our theories as a \emph{Fourier duality}.  This also admits a diagrammatic description, but we prove that duality preserves expectation values using a Plancherel's theorem at each regularized level.  It is worth remarking that there are three different ``levels'' of factorization algebra necessary to make sense of Fourier duality for observables in the generalized Maxwell theory.  The dual itself is defined for a theory where the fields consist of all $p$-forms, but at this level duality doesn't preserve expectation values.  An observable in this theory restricts to an observable in a theory where the fields consist of only \emph{closed} $p$-forms and at this level duality \emph{does} preserve expectation values.  However in order to define a dual we now need to \emph{choose} an extension from an observable acting on closed $p$-forms to an observable acting on all $p$-forms.  We can phrase this in terms of a correspondence of prefactorization algebras: observables are called \emph{incident} if they are the images of the same observable under a pair of restriction maps (to the closed $p$-form theory and its dual).  The third level is that of the generalized Maxwell theory we're really interested in.  On an open set we can construct a map from observables in the closed $p$-form theory to observables in the generalized Maxwell theory, which is an isomorphism of local sections of the prefactorization algebra if the open set is contractible, for instance.  This gives us a way of defining a dual of a local observable in the original generalized Maxwell theory.

It would be interesting to investigate extensions to twisted supersymmetric abelian (higher) gauge theories.  In particular an understanding of abelian duality for topological twists of maximally supersymmetric abelian gauge theories, which is sufficiently complete to include an understanding of boundary conditions, should -- according to the approach of Kapustin and Witten in \cite{KW} -- recover the theory of \emph{geometric class field theory}.  Specifically one expects an equivalence of the categories of branes which should recover the abelian version of the geometric Langlands correspondence, as realized independently by Laumon \cite{Laumon} and Rothstein \cite{Rothstein} by a twisted Fourier-Mukai transformation.

\subsection*{Acknowledgements}
I would like to thank Kevin Costello for many helpful ideas and discussions throughout this project, and Thel Seraphim for some of the initial ideas on how to view abelian duality for expectation values as a version of Plancherel's theorem.  I would also like to thank Saul Glasman, Sam Gunningham, Owen Gwilliam, Boris Hanin, Theo Johnson-Freyd, David Nadler, Toly Preygel, Nick Rozenblyum and Jesse Wolfson for helpful conversations, and the anonymous referees for many useful comments and corrections.  Finally, I'd like to thank Theo Johnson-Freyd, Aron Heleodoro and Philsang Yoo for carefully reading an earlier draft and offering many helpful comments and corrections.  Figures were created using the diagramming software \emph{Dia}.

\section{The BV Formalism for Free Field Theories} \label{sectionBV}
\subsection{The Idea of the BV Formalism} \label{BV_informal}
The \emph{Batalin-Vilkovisky formalism} (hereafter referred to as the BV formalism) \cite{BatalinVilkovisky} gives a description of the moduli space of solutions to the equations of motion in a classical field theory that is particularly amenable to quantization.  When we quantize following the BV recipe, we will see -- in the case of a free theory -- that the Feynman path integral description of the expectation values of observables naturally falls out.  This quantization procedure admits an extension to interacting theories: see \cite{CostelloBook} and \cite{Book2} for details.

We start with a rough description and motivation of the classical BV formalism.  In its simplest form, a classical field theory consists of a space $\Phi$ of \emph{fields} (often the global sections of a sheaf over a manifold $X$ which we call \emph{spacetime}), and a map $S \colon \Phi \to \RR$: the \emph{action functional}.  The physical states in this classical system are supposed to be those states which extremise the action, i.e. the \emph{critical locus} of $S$: the locus in $\Phi$ where $\d S=0$.  This can be written as an intersection, specifically as
\[\crit(S) = \Gamma_{\d S} \cap X\]
where $\Gamma_{\d S}$ is the graph of $\d S$ in the cotangent bundle $T^*X$, and $X$ is the zero section.  The classical BV formalism gives a model for functions on the \emph{derived} critical locus of $S$: that is, more than just forming the pullback in spaces given by this intersection, one forms a derived pullback in a homotopy category of spaces, and considers its ring of functions.

We can describe the ring of functions on the derived critical locus explicitly as a derived tensor product by resolving $\OO(\Gamma_{\d S})$ as an $\OO(T^*X)$-module.  We choose the Koszul resolution.  Explicitly this says that there is a quasi-isomorphism:
\[\OO(\Gamma_{\d S}) \iso \left(\xymatrix{\cdots \ar[r] &\OO(T^*X) \otimes_{\OO(X)} \bigwedge^2 T_X \ar[r] &\OO(T^*X) \otimes_{\OO(X)} T_X \ar[r] &\OO(T^*X)}\right)\]
where $T_X$ denote the module of vector fields, and the differential is extended from the map $\OO(T^*X) \otimes_{\OO(X)} T_X \to \OO(T^*X)$ sending $f \otimes v$ to $fv - f\iota_{\d S}(v)$ as a derivation with respect to the wedge product  Taking this complex and tensoring with $\OO(X)$ we find the complex $\PV(X)$ of \emph{polyvector fields} on $X$, i.e. exterior powers of the ring of vector fields placed in non-positive degrees, with the differential $-\iota_{\d S}$ from vector fields to functions extended to a differential on the whole complex as a derivation for the wedge product.  This model for functions on the derived critical locus is the BV model for the algebra of \emph{classical observables} in the Lagrangian field theory.

\begin{remark}
We should remark that in this paper we'll work in a very restricted context, that of free theories, and therefore we will not need to use sophisticated notions from the theory of derived algebraic geometry.
\end{remark}

Now, we motivate the \emph{quantum} BV formalism by means of a toy example.  Let $\Phi$ be a finite-dimensional vector space, and let $S$ be a quadratic form on this vector space.  In this toy example, quantum field theory (in Euclidean signature) boils down to the computation of the Gaussian integrals
\begin{align*}
\langle \OO \rangle &= \frac {\int_\Phi \OO(\phi) e^{-S(\phi)/\hbar} d\phi} {\int_\Phi e^{-S(\phi)/\hbar} d\phi} \\
&= \frac 1Z \int_\Phi \OO(\phi) e^{-S(\phi)/\hbar} d\phi
\end{align*}
for polynomial functions $\OO$ on $\Phi$ (writing $Z$ for the normalising factor $\int e^{-S(\phi)/\hbar} d\phi$).  Here $\hbar$ is a positive real number and $\d\phi$ is a volume form on $\Phi$.  Equivalently, we can think of this as computing the cohomology class of a top degree element $\OO d\phi$ in a twisted de Rham complex: the complex of polynomial differential forms $\Omega^*_{\text{poly}}(\Phi)$ with differential $\d - \frac 1 \hbar(\wedge \d S)$.  Finally, contracting with the top form $\d\phi$ gives an isomorphism of graded vector spaces $\PV(\Phi)[-\dim \Phi] \to \Omega^*(\Phi)$, which becomes an isomorphism of \emph{complexes} when one gives the space of polyvector fields the differential $D - \frac 1\hbar \iota_{\d S}$, where $D$ is the \emph{BV operator} given by transferring the exterior derivative along the map $\iota_{d\phi}$.  Concretely, let $x_1, \ldots x_n$ form a basis for $\Phi$, and let $\dd_1, \ldots, \dd_n$ be the corresponding basis on the tangent space $T_0\Phi$.  Say $\d\phi = \d x_1 \wedge \cdots \d x_n$.  Then
\[D = \sum_{i=1}^n \frac \dd{\dd x_i} \frac \dd {\dd(\dd_i)}.\]

If $\Phi$ is infinite-dimensional then we can no longer immediately make sense of the original Gaussian integral (though we can compute it as a suitable limit), nor of ``top'' degree forms in the twisted de Rham complex.  But the complex $\PV(\Phi)$ in degrees $\le 0$ and the differential $D - \frac 1\hbar\iota_{\d S}$ still makes sense, and we can still compute the cohomology class of a degree zero element, thus defining its \emph{expectation value} directly.  What's more, we see that, considering instead the isomorphic complex with differential $\hbar D - \iota_{\d S}$, in the ``classical'' limit as $\hbar \to 0$ we recover the BV description of the algebra of classical observables.  So this explicitly gives a \emph{quantization} of that algebra.  This quantization is no longer a dg-algebra: the Leibniz identity receives a correction term proportional to $\hbar$ coming from the classical Poisson bracket.

The general BV formalism therefore gives a model for the classical and quantum observables in a free Lagrangian field theory (i.e. a theory with quadratic action) following this outline.  The classical observables are constructed as an algebra of polyvector fields with the differential $\iota_{\d S}$, and a quantization is produced by deforming this differential with a BV operator analogous to the one above.  One builds this operator by -- approximately -- identifying Darboux co-ordinates on the (shifted) cotangent bundle to the fields and defining an operator using a formula like the one given above.  An easier way to describe this is to use the Poisson bracket on functions on the shifted cotangent bundle (the so-called \emph{antibracket}), and to extend this to a BV operator by an inductive formula on the degree. 

\subsection{Free Fields with Sectors} \label{sectionDerived}
When defining a classical field theory on a manifold $X$, it's not completely clear what kind of object one should use to define the ``sheaf of fields'' of a Lagrangian field theory on $X$.  One can build a classical field theory starting from a cochain complex of vector spaces (this is the approach used by Costello and Gwilliam \cite{Book2}).  Many interesting examples, such as non-linear sigma models, would be better served by a formalism where the fields could take values in a space with non-trivial topology.  We would also like to be able to include discrete data in the space of fields, for instance the choice of $G$-bundle for fields in Yang-Mills theory.  Witten's work on abelian duality \cite{WittenAbelian} shows that this discrete invariant is necessary for the existence of duality phenomena: one sees theta functions in the partition function of an abelian gauge theory only after summing over all topological sectors.

In this section we'll use the following definition of a classical Lagrangian field theory which, while not the most general definition possible, allows for discrete pieces in the fields suitable for the free theories we will consider. 

\begin{remark}
All algebraic constructions with topological vector spaces in this paper take place in the context of nuclear Frech\'et spaces (or cochain complexes thereof).  For instance, the dual space $V^\vee$ of a vector space $V$ is always the continuous dual equipped with the strong topology, and the tensor product is the completed projective tensor product.  
\end{remark}

\begin{definition} \label{Lag_theory_def}
A \emph{classical Lagrangian pre-theory with sectors} on a manifold $X$ consists of a datum $(\Phi, \LL)$, where
\begin{enumerate}
 \item The \emph{fields} $\Phi$ are a presheaf of cochain complexes of topological abelian groups $\Phi$ equipped with the structure of an extension
 \[0 \to V(U) \to \Phi(U) \to \Phi_{\sect}(U) \to 0\]
 on each open set $U \sub X$ compatible with the restriction maps, where $V(U)$ is a cochain complex concentrated in degrees $\le 0$, and $\Phi_\sect(U)$ is a discrete abelian group.
 \item The \emph{Lagrangian density} $\LL$ is a morphism of presheaves on $X$
 \[\mc L \colon \Phi \to \ul{\dens}\]
where $\dens$ denotes the sheaf of densities on $X$, viewed as a sheaf of cochain complexes.
\end{enumerate}
The pre-theory $(\Phi,\LL)$ is a classical Lagrangian \emph{theory} with sectors if $\Phi$ is a sheaf, and $V$ is the sheaf of sections of a complex of vector bundles.

We say a Lagrangian pre-theory is \emph{linear} if $\Phi_\sect(U) = \pt$ for all $U$.
\end{definition}

\begin{remark}
It's worth noting that we could also define a theory with fermions by allowing $\Phi$ to be instead a sheaf of supergroups.  We won't need this generality for the examples of this paper.
\end{remark}

\begin{remark} \label{derived_remark}
There's an alternative, and more general approach, that we could take to defining theories where the fields include non-linear data, where we would model the fields as a sheaf of \emph{derived stacks} (using theory developed in particular by To\"en-Vezzosi \cite{HAGI, HAGII} and Lurie \cite{HigherAlgebra}).  A \emph{prestack} is a functor of $\infty$-categories $Z \colon \mr{cdga}_{\le 0}^{\mr{op}} \to \mr{sSet}$, where $\cdga_{\le 0}$ is the $\infty$-category of commutative dgas over $\CC$ concentrated in cohomological degrees $\le 0$, and $\sset$ is the $\infty$-category of simplicial sets.  One could define an \emph{abelian group prestack} to likewise be a functor of $\infty$-categories $Z \colon \mr{cdga}_{\le 0}^{\mr{op}} \to \mr{sAb}$, where $\mr{sAb}$ is the $\infty$-category of simplicial abelian groups.

The \emph{Dold-Kan correspondence} \cite[Theorem 1.2.4.1]{HigherAlgebra} provides us with an equivalence of $\infty$-categories
\[\mr{DK} \colon \mr{sAb} \to \mr{Ch}_{\le 0},\]
where $\mr{Ch}_{\ge 0}$ is the $\infty$-category of cochain complexes of abelian groups concentrated in degrees $\le 0$.  We could therefore define a Lagrangian theory whose fields are modelled by a sheaf of functors to the $\infty$-category $\mr{Ch}_{\ge 0}$.

Our main example (Example \ref{delignecochains} below) can be defined as a functor of points in this way.  To really work with this formalism however, it's not enough to work in the generality of all prestacks.  In order to make sense of classical observables we would need a well-behaved cotangent complex, so for instance it would be enough for the fields to be modelled by a derived Artin stack (see e.g. \cite{ToenSurvey} for an explanation of this idea).  By working only with classical Lagrangian theories with sectors as in Definition \ref{Lag_theory_def} we avoid having to address this issue in this paper.
\end{remark}

We'll construct examples of classical Lagrangian theories with sectors motivated by the following idea.

\begin{example}[Yang-Mills Theory]
We begin with an informal example, to motivate our main example \ref{delignecochains} below.  Let $G$ be a compact connected Lie group.  Define a presheaf $\Phi$ of abelian groups $X$ modelling the stack of connections on principal $G$-bundles: for an open set $U \sub X$ we set
\[\Phi(U) = \bigoplus_P \left( \Omega^0(U; \gg_P)[1] \to \Omega^1(U; \gg_P)  \right)\]
where the sum is over principal $G$ bundles $P \to U$ up to isomorphism.  We emphasise that this is an informal example for motivation only.  Indeed, in general this does not manifestly constitute an example of, even, a pretheory, since the complexes $\Omega^0(U; \gg_P) \to \Omega^1(U; \gg_P)$ are not isomorphic for different choices of $P$.

If $G$ is abelian then each summand models the space of connections on the bundle $P$ modulo the action of the algebra of infinitesimal gauge transformations of $P$ (for $G$ non-abelian we would need to account additionally for the data of the bracket).  We'll use a higher rank generalization of this construction.
\end{example}

The main objects of study in this paper are a family of theories generalising Yang-Mills theory with gauge group $U(1)$, following the description of generalized Maxwell theory through ordinary differential cochains described in \cite{Freed}.  The fields in this theory should describe ``circle $(p-1)$-bundles with connection'', or ``circle $(p-2)$-gerbes with connection'' for $p$ some positive integer (the reason we use $p-1$ will become clear later: the ``curvature'' of a field in such a theory will be a $p$-form).  

\begin{definition} \label{delignecochains}
The \emph{smooth Deligne complex} on a manifold $X$ is the sheaf of cochain complexes of abelian groups with sections on an open set $U$ given by
\begin{align*}
\ZZ(p)_{\mc D}(U) \xymatrix@R-12pt{= &\ZZ[p]\, \ar@{^{(}->}[r] &\Omega^0(U)[p-1] \ar[r]^d &\Omega^1(U)[p-2] \ar[r]^(0.7)d &\cdots \ar[r]^(0.35)d &\Omega^{p-1}(U) \\
\iso &&C^\infty(U, \RR/\ZZ)[p-1] \ar[r]^{-i d\log}  &\Omega^1(U)[p-2] \ar[r]^(0.7)d &\cdots \ar[r]^(0.35)d &\Omega^{p-1}(U)}.
\end{align*}
The \emph{complexified} Deligne complex has local sections given by
\[\ZZ(p)_{\mc D,\CC}(U) \xymatrix@R-12pt{= &C^\infty(U, \CC^\times)[p-1] \ar[r]^{-i d\log}  &\Omega^1(U; \CC)[p-2] \ar[r]^(0.7)d &\cdots \ar[r]^(0.35)d &\Omega^{p-1}(U; \CC)}.\]
We think of these complexified fields as higher principal $\CC^\times$ bundles with connection, or as higher complex line bundles with connection.
\end{definition}

\begin{definition}
The $p^\text{th}$ \emph{Deligne cohomology} group $\hat H^p(X)$ of a smooth manifold $X$ is defined to be the  $0^\text{th}$ hypercohomology group of the Deligne complex.
\end{definition}

\begin{remark} \label{Deligne_hypercohomology}
The degree $p$ differential cohomology group $\hat H^p(U)$) has a natural description.  We can compute it using the long exact sequence on hypercohomology associated to the short exact sequence of sheaves
\begin{equation} 
\label{ex_seq_formula} 0 \to \Omega^{\le p-1}_\CC[p-1] \to \ZZ(p)_{\mc D,\CC} \to 2 \pi R \ZZ[p] \to 0,
\end{equation}
which will tell us (see \cite[Theorem 1.5.3]{Brylinski}) that $\bb H^0(U; \ZZ(p)_{\mc D})$ is isomorphic to the product of a torus (on which there are no non-constant global functions, so it won't contribute to the local observables of our classical field theory) and the group $\Omega^p_{cl, \ZZ}(U; \CC)$ defined as follows.
\end{remark}

\begin{definition} \label{integral_periods_def}
The group $\Omega^p_{cl, \ZZ}(U; \CC)$ of closed $p$-forms with \emph{integral periods} is the subgroup of the group $\Omega^p_{\mr{cl}}(U;\CC)$ of closed $p$-forms comprising those forms $\alpha$ whose cohomology class $[\alpha]$ lies in the subgroup $\mr H^p(U; 2\pi R \ZZ) \sub \mr H^p(U; \CC)$. If $U$ is compact and oriented then this space splits as the product of an infinite-dimensional vector space and a finite rank lattice. 
\end{definition}

We'll need one more ingredient to define the action functional of our classical field theory: the notion of the curvature of a Deligne cochain.
\begin{definition}
The \emph{curvature map} is the map of presheaves of cochain complexes $F \colon \ZZ(p)_{\mc D} \to \Omega^p_{\cl}$ induced from the de Rham differential.
\end{definition}

Let us now introduce our main example of a classical Lagrangian field theory with sectors: we will define a classical field theory whose fields are given by the Deligne complex $\ZZ(p)_{\mc D}$. The \emph{action} on a compact manifold will be defined to be the $L^2$-norm of the curvature of a field.

\begin{definition}[Generalized Maxwell Theory] 
The \emph{generalized Maxwell theory} of degree $p$ on an $n$-manifold Riemannian manifold $X$ is the classical Lagrangian field theory $(\Phi^\Max_p, \LL_R)$ with sectors, where $\Phi^\Max_p = \ZZ(p)_{\mc D}$ is the Deligne complex, and $\LL_R$ is the density-valued functional
\[\mc L_R(\phi) = R^2 F(\phi) \wedge \ast F(\phi),\]
where $R$ is a positive real number and $\ast$ is the Hodge star associated to the metric on $X$.  The sheaf of fields can be written as an extension, as required by Definition \ref{Lag_theory_def}, using the exact sequence \ref{ex_seq_formula}, so that $\Phi_{\mr{sect}}$ is just $2 \pi R \ZZ[p]$, a rank one lattice placed in cohomological degree $-p$.  Note that $\Phi_{\mr{sect}}$ alone does not form a sheaf.
\end{definition}

If this density $\LL_R(\phi)$ is integrable, then the resulting integral is just the $L^2$-norm of $F(\phi)$.  We call the number $R$ here a \emph{coupling constant}, and think of it as the radius of the gauge group circle, pre complexification.  We could also have produced this scaling by redefining the fields: in our Deligne complex we might have included the lattice $2\pi R\ZZ$ instead of $\ZZ$, yielding a circle of radius $R$ in the cohomology.

\begin{remark}
We can generalize this setting from a circle (or, in the complexified story, $\CC^\times$) to a higher rank torus $T \iso V/\Lambda$, where $\Lambda$ is a full rank lattice in a real normed vector space $V$.  Let $V_\CC = V \otimes_\RR \CC$ and $T_\CC$ denote the complexifications of $V$ and $T$.  The fields in the theory described above generalize immediately by taking algebraic $T_\CC$ valued functions mapping into the algebraic de Rham complex with values in $V_\CC$.  There is now a curvature map taking values in $\Omega^p_{\cl} \otimes V_\CC$, and we define a Lagrangian density functional by
\[\mc L_\Lambda(\phi) = \|F(\phi) \wedge \ast F(\phi)\|^2\]
where $\|-\|$ is the norm on $V_\CC$.

In this description, the coupling constants arise from the choice of lattice $\Lambda \sub V$: for instance we obtain the theory with coupling constant $R$ above by choosing the lattice $2\pi R \ZZ \sub \RR$.
\end{remark}

\subsection{The Action Functional and the Classical Prefactorization Algebra} \label{sectionAction}
Once we have the sheaf of fields $\Phi$ we can apply the classical BV procedure to build a model for the derived critical locus of the action that is amenable to quantization.  The first step is to describe the shifted cotangent bundle $T^*[-1]\Phi$ as a derived stack, or -- to avoid requiring too much formalism from derived algebraic geometry -- to directly describe the algebra of functions $\OO(T^*[-1]\Phi)$ as a cochain complex.  We have some freedom to choose exactly how much our classical observables depend on the discrete datum of $\Phi_\sect$ -- we've chosen to allow the most general behavior, where the observables can act differently on each sector in an unconstrained way.

\begin{definition}
If $V$ is a cochain complex of topological vector spaces \footnote{More precisely, we usually require $V$ to be a \emph{differentiable vector space} as in \cite[Appendix B]{Book1}, in order to have well behaved homological algebra.  For instance, the sections of a complex of smooth vector bundles satisfy this condition.}, we define $\OO(V)$ to be the complex $\sym^\bullet(V^\vee)$, where $V^\vee$ is the continuous linear dual of $V$, with differential extended to the symmetric algebra as a derivation.
\end{definition}

\begin{definition}
Let $V \to \Phi \to \Phi_\sect$ be the presheaf of fields in a classical Lagrangian theory with sectors as in Definition \ref{Lag_theory_def}).  Let $\OO(\Phi)$ denote the cochain complex
\[\OO(\Phi) = \prod_{\mr H^0(U;\Phi_\sect)} \OO(V(U)).\]
Likewise, let $\OO(T^*[-1]\Phi(U))$ denote the cochain complex 
\[\OO(T^*[-1]\Phi(U)) = \prod_{\mr H^0(U;\Phi_\sect)} \OO(V(U)) \otimes \sym(V(U)[1]).\]
\end{definition}

\begin{remark}
A more sophisticated definition, using ideas from derived algebraic geometry (see Remark \ref{derived_remark}) to define the ring of functions $\OO(\Phi)$ on a derived stack of fields, would incorporate not only $\mr H^0$ of the presheaf $\Phi_\sect$, but the full derived object, including its cohomology in non-zero degrees.  In this work we have only used this simpler definition, in order to avoid having to work with these more subtle and difficult objects.
\end{remark}

This definition gives the ring $\OO(T^*[-1]\Phi)$ a natural interpretation as a ring of \emph{polyvector fields} on $\Phi$.  Indeed, we can identify $V$ with $T_0\Phi$: the tangent complex to $\Phi$ at the identity.  Then the dg-vector space $T_0\Phi \otimes \OO(\Phi)$ precisely describes vector fields on $\Phi$.  Placing this space in degree $-1$ and taking graded symmetric powers (i.e. alternating powers, by the usual sign rule), we produce the algebra of polyvector fields on $\Phi$.

Now we incorporate the action.  Recall that as well as the fields, our Lagrangian field theory data included a map of sheaves from $\Phi$ to the sheaf of densities on $X$.  While it is not, in general, possible to integrate the resulting local densities, it \emph{is} possible to define the first variation of this ``local action functional''.  

\begin{definition}
Let $(\Phi, \LL)$ be a Lagrangian theory with sectors.  Let $V_c$ denote the sheaf of compactly supported sections of $V$.  We define a compactly supported 1-form $\d S$ on $\Phi$, i.e. an element of $V_c^\vee \otimes \OO(\Phi) \iso \hom(V_c , \OO(\Phi))$ as
\begin{align*}
\d S \colon V_c &\to \OO(\Phi) \\
v &\mapsto \left(\phi \mapsto \int_U L_v(\LL(\phi)) \right),
\end{align*}
where $L_v(\LL(\phi))$ denotes the Lie derivative of the density $\LL(\phi)$ along the vector field $v$.  The compact support condition ensures that $L_v(\LL(\phi))$ is a compactly supported density, so that the integral is well-defined.
\end{definition}

The action functional $S$ describes a modified version of the shifted cotangent bundle by modifying the internal differential on its ring of functions.  After identifying the functions on the shifted cotangent bundle with polyvector fields, the 1-form $\d S$ naturally defines a degree one linear operator, namely the interior product
\[\iota_{\d S} \colon \OO(T^*[-1]\Phi) \to \OO(T^*[-1]\Phi).\]
More explicitly, the operator is extended as a derivation from the operator $V_c \otimes \OO(\Phi) \to \OO(\Phi)$ given by pairing a vector field with the 1-form $\d S \in V_c^\vee \otimes \OO(\Phi)$.

\begin{remark}
If the theory $(\Phi, \LL)$ is linear, i.e. if $\Phi_\sect$ is trivial, then we can describe the complex $\OO(T^*[-1]\Phi)$ of polyvector fields and the classical differential $\iota_{\d S}$ more concisely.  The global functions now form a symmetric algebra
\[\OO(T^*[-1]\Phi) \iso \sym(\Phi^\vee \oplus \Phi[1])\]
generated by linear functions and linear vector fields on $\Phi$.  We'll see this later in some examples of free theories, where the action is encoded by a linear operator $\Phi \to \Phi^\vee$ of degree one.
\end{remark}

The above discussion took place for a fixed open set $U \sub X$.  Let's now describe the relationship between the classical observables on different open sets.  We describe locality using the machinery of \emph{factorization algebras}, as developed in \cite{Book1}.  We recall the basic definitions.

\begin{definition}[{\cite[Chapter 3, Chapter 6]{Book1}}]
A \emph{prefactorization algebra} $\mc F$ on a space $X$ taking values in a symmetric monoidal category $\mc C$ with small colimits is a $\mc C$-valued precosheaf on $X$ equipped with $S_k$-equivariant isomorphisms 
\[\mc F(U_1) \otimes\cdots \otimes \mc F(U_k) \to \mc F(U_1 \sqcup \cdots \sqcup U_k)\]
for every collection $U_1, \ldots, U_k \sub X$ of disjoint open sets.

An open cover $\{U_i\}$ of a space $X$ is called \emph{factorizing} or \emph{Weiss} if for every finite subset $\{x_i, \ldots, x_\ell\}$ of $X$ there is a collection $U_{i_1}, \ldots, U_{i_\ell}$ of \emph{pairwise disjoint} sets in the cover such that $\{x_i, \ldots, x_\ell\} \sub U_{i_1} \cup \cdots \cup U_{i_\ell}$.

Given an open cover of $X$ and a precosheaf $\mc F$ on $X$ we can construct a simplicial object in $\mc C$ called the \emph{\v Cech complex} of $\mc F$, defined as
\[\check {C}(U, \mc F) = \bigoplus_{k=1}^\infty \left( \bigoplus_{U_{i_1}, \ldots, U_{i_k}} F(U_{i_1} \cap \cdots \cap U_{i_k})[k-1] \right)\]
with the usual \v Cech maps.

A prefactorization algebra is a \emph{factorization algebra} if for every open set $U \sub M$ and every factorizing cover $\{U_i\}$ of $U$, the natural map $\colim \check {C}(U, \mc F) \to \mc F(U)$ is an isomorphism in $\mc C$.
\end{definition}

\begin{definition}
The \emph{prefactorization algebra of classical observables} associated to the classical Lagrangian theory $(\Phi, \mc L)$ is the factorization algebra $\obscl_\Phi(U)$ valued in cochain complexes whose sections on $U$ are given by the complex $\OO(T^*[-1]\Phi)$ with differential given by the internal differential $\d_{\Phi}$ on $\OO(\Phi)$, plus the classical differential $-\iota_{\d S}$.
\end{definition}

\begin{remark}
If the theory is linear, then this construction defines a factorization algebra, rather than just a prefactorization algebra.  This fact appears as Theorem 4.5.1 in \cite{GwilliamThesis}.  It follows from the fact that $\Phi$ forms a sheaf, so the global functions $\OO(\Phi)$ forms a cosheaf.  This, however, does not hold when we introduce sectors: because $\mr H^0(-; \Phi_\mr{sect})$ does not generally define a sheaf, we cannot recover observables on open sets $U$ where $\Phi_{\mr{sect}}$ has interesting cohomology purely from observables on contractible open sets. 
\end{remark}

Finally, we need to address the \emph{Poisson structure} on the classical observables.  This is Definition 2.1.3 in \cite{GwilliamThesis}.

\begin{definition}
A \emph{$P_0$-prefactorization algebra} is a prefactorization algebra $\mc F$ valued in cochain complexes, such that each $\mc F(U)$ is equipped with a commutative product and a degree 1 antisymmetric map $\{,\} \colon \mc F(U) \otimes \mc F(U) \to \mc F(U)$ which is a biderivation for the product, satisfies the identity
\[\d\{x,y\} = \{\d x,y\} + (-1)^{|x|} \{x,\d y\}\]
and is compatible with the prefactorization structure.
\end{definition}

We expect such a structure on the classical observables coming from the \emph{shifted symplectic} structure on the shifted cotangent bundle, but we can recall a familiar concrete description (the well-known Schouten bracket on polyvector fields).  Firstly, there's an evaluation map defined as 
\begin{equation}
\label{evaluation_eqn}
\ev \colon V \otimes \OO(\Phi) \to \OO(\Phi)
\end{equation}
taking an element $v \otimes f$ to $\d f(v) \in \OO(\Phi)$ (thinking of the tangent vector $v \in V = T_0\Phi$ as a constant vector field on $\Phi$).  We use this to define the Schouten bracket in low polyvector field degrees:
\begin{align*}
\{1 \otimes f_1, 1 \otimes f_2 \} &= 0 \\
\{v_1 \otimes f_1, 1 \otimes f_2 \} &= 1 \otimes \d f_2(v_1) \cdot f_1 \\
\{v_1 \otimes f_1, v_2 \otimes f_2 \} &= v_2 \otimes df_2(v_1) \cdot f_1 - v_1 \otimes \d f_1(v_2) \cdot f_2
\end{align*}
This extends uniquely to an antisymmetric degree 1 pairing on the whole algebra of polyvector fields as a biderivation with respect to, as usual, the wedge product of polyvector fields.

\subsection{Quantization of Free Prefactorization Algebras} \label{quantization_section}
From now on we will restrict attention to \emph{free} field theories, where we can use the intuitive, non-perturbative notion of BV quantization described in Section \ref{BV_informal}.  Informally,  a classical Lagrangian field theory is free if the action functional is \emph{quadratic}, so the derivative of the action functional is \emph{linear}.

\begin{definition}
 A classical Lagrangian field theory is called \emph{free} if the classical differential $\iota_{\d S}$ increases polynomial degrees by one.  That is, if we filter each factor $\OO(V)_P$ in $\OO(\Phi)$ for an element $P \in \Phi_\sect$ by polynomial degree and call the $k^{\text{th}}$ filtered piece $F^k\OO(\Phi)_P$, the operator $\iota_{\d S}$ raises degree by one:
 \[\iota_{\d S} \colon \sym^i(V[1]) \otimes F^j\OO(\Phi)_P \to \sym^{i-1}(V[1]) \otimes F^{j+1}\OO(\Phi)_P.\]
\end{definition}

Now, let $(\Phi, \mc L)$ be a free classical theory with sectors, and let $\obscl(\Phi(U))$ be the complex of classical observables on an open set $U$.  We'll quantize the local observables by adding a new term to the differential on this complex: the \emph{BV operator}, which we'll denote by $D$.  This operator is built from the $P_0$-algebra structure on the classical observables, following the method of deformation quantization for free theories described in \cite{GwilliamThesis}.

\begin{definition}
The \emph{BV operator} $D \colon \OO(T^*[-1]\Phi(U)) \to \OO(T^*[-1]\Phi(U))$ is determined as follows.  Set $D$ to be zero on $\OO(\Phi(U))$, and to be given by the Poisson bracket in polynomial degree 1: $D = \{,\} \colon T_0\Phi(U) \otimes \OO(\Phi(U)) \to \OO(\Phi(U))$, i.e. the map we described in equation \ref{evaluation_eqn} above as ``evaluation''.  We can then extend this to an operator on the whole complex of classical observables according to the formula
\begin{equation} \label{BD_formula}
D(\phi \cdot \psi) = D(\phi) \cdot \psi + (-1)^{|\phi|} \phi \cdot D(\psi) + \{\phi, \psi\}.\end{equation}
\end{definition}

An algebra with a differential $D$ and Poisson bracket $\{,\}$ satisfying a formula like \ref{BD_formula} is called a \emph{Beilinson-Drinfeld algebra}, or \emph{BD algebra}: Beilinson and Drinfeld constructed in \cite{BD} a family of operads over the formal disc whose fibre at the origin is the $P_0$ operad.  The BD algebra structure given here is a description of an algebra for a generic fibre of the analogous family defined over all of $\CC$ rather than just a formal neighbourhood of the origin.

\begin{example}
If $\Phi$ is linear, so that the classical observables are given by $\sym(\Phi[1] \oplus \Phi^\vee)$, then we can construct the BV operator even more directly.  The Poisson bracket, restricted to $\sym^{\le 2}$ is given by the evaluation map $\Phi \otimes \Phi^\vee \to \CC$ from $\sym^2$ to $\sym^0$, and zero otherwise.  This extends uniquely to a degree 1 operator on the whole complex of classical observables, lowering $\sym$ degree by 2 as a BD structure, as above.
\end{example}

Equipped with this operator we can now define the quantum observables.
\begin{definition}
The prefactorization algebra of \emph{quantum observables} for the free theory $(\Phi, \mc L)$ with sectors is the prefactorization algebra with local sections on $U$ given by the cochain complex
\[\obsq(U) = (\OO(T^*[-1]\Phi(U)), \d_\Phi - \iota_{\d S} + D)\]
where $\d_\Phi$ is the differential coming from the internal differential on $\Phi$, $\iota_{\d S}$ is the classical BV differential, and $D$ is the quantum BV differential as defined above.
\end{definition}
In the linear case, the fact this procedure actually defines a factorization algebra is proved in \cite{GwilliamThesis} (see also \cite[Section 9]{Book2}).

\begin{remark}
A more standard thing to write would consider a differential $\d_\Phi - \iota_{\d S} + \hbar D$, and have the BV operator $D$ defined by a formula like
\[D(\phi \cdot \psi) = D(\phi) \cdot \psi + (-1)^{|\phi|} \phi \cdot D(\psi) + \hbar \{\phi, \psi\}\]
where we'd adjoined a formal parameter $\hbar$ to the algebra of classical observables.  We'd then obtain a module flat over $\RR[[\hbar]]$ which recovered the classical observables upon setting $\hbar$ to zero.  We haven't done this here, because when working exclusively with free theories it's possible to work completely \emph{non-perturbatively}, i.e. to evaluate at a non-zero value of $\hbar$.  The duality phenomena we're investigating are only visible non-perturbatively (taking into account all topological sectors) so this is necessary, however we lose the ability to consider the quantum observables in any theory which is not free.
\end{remark}

\subsection{Smearing Observables} \label{smearing}
In order to describe expectation values in Section \ref{sectionExpVal} we need to identify a dense subprefactorization algebra of the classical observables which are generated by polynomials in the fields, as opposed to their (distributional) dual spaces.  To do so, we restrict to the setting where $(\Phi, \LL)$ is linear, so the algebra of functions $\OO(\Phi(U))$ on the fields is a free commutative dg algebra, and
\[\OO(T^*[-1]\Phi(U)) = \sym(\Phi(U)^\vee \oplus \Phi(U)[1]).\]

Now, suppose further that the theory is free.  Then as an operator on this symmetric algebra the classical differential $\iota_{\d S}$ is non-increasing in $\sym$-degree.  It can be made to \emph{preserve} $\sym$-degree by completing the square in the action functional, in order to eliminate linear terms, and therefore it can be described as the extension of its \emph{linear} part, which we'll denote
\[Q \colon \Phi(U) \to \Phi(U)^\vee\]
to the whole complex, as a derivation.  Thus the classical observables are themselves given by the free cdga on a complex, namely the complex $\Phi(U)[1] \overset Q \to \Phi(U)^\vee$.  The dual space $\Phi(U)^\vee$ is ``bigger'' than the space $\Phi(U)$: it includes distributional sections of the underlying dual complex of vector bundles.  The procedure we'll describe in this section will replace it with a space of ordinary, smooth sections of a bundle.

The \emph{smeared} or \emph{smooth} classical observables form a dense subalgebra of $\obscl(U)$, defined using the data of an \emph{invariant pairing} on the space of fields.  We proceed in two steps.  

\begin{enumerate}
\item First, recall that $\Phi$ is obtained as the sheaf of sections of a complex $(E, \d)$ of vector bundles.  There is a dense subspace of $\Phi(U)^\vee$ given by the space of local sections of the dual complex $(E^*, \d^*)$ of vector bundles

\item Next, suppose that we're given an antisymmetric pairing on the underlying graded vector bundle $E^\natural$ of $E$ obtained by forgetting the differential:
\[\langle -,- \rangle \colon E^\natural \otimes E^\natural \to \CC,\]
taking values in the trivial bundle, which is non-degenerate as a scalar-valued pairing on the fibers.  This pairing allows us to identify the dual complex $E^*$ as the bundle $E^\natural$, with reversed grading, and equipped with the adjoint differential operator.
\end{enumerate}

\begin{definition}\label{smeared_bundle_def}
Let $\Phi^\text{sm}$ be the cosheaf of compactly supported sections of the vector bundle $E^\text{sm}$, with $E^\text{sm}_i = E_{-i}$, and with differential adjoint to the differential on $E$.
\end{definition}

The pairing defines a canonical embedding
\[\Phi^\text{sm}(U) \to \Phi^\vee(U)\]
with dense image in each degree, sending a compactly supported vector field to the functional ``pair with that vector field''. 

\begin{definition}
Suppose the classical map $Q \colon \Phi_c(U) \to \Phi(U)^\vee$ factors through a linear map $Q^{\text{sm}} \colon \Phi_c(U) \to \Phi^{\text{sm}}(U)$.  The classical \emph{smeared} or \emph{smooth} observables on an open set $U$ are defined to be the free cdga
\[\obssmcl_\Phi(U) = \sym(\Phi_c(U)[1] \overset {Q^{\text{sm}}} \to \Phi^{\text{sm}}(U)).\]
If $\OO(\Phi(U))$ is a free cdga for every open set $U \sub X$ then the smooth observables define a subprefactorization algebra of $\obscl_\Phi$, dense in every degree.
\end{definition}

\begin{definition} \label{elliptic_theory_def}
We call a free classical field theory equipped with an invariant pairing \emph{elliptic} if the resulting complex of \emph{linear} smeared observables 
\[\mc E(X) = \Phi^\text{sm}(X)[1] \overset {Q^{\text{sm}}} \to \Phi_c(X)\]
on the total space of the manifold $X$ arises as the sheaf of global sections of an elliptic complex of vector bundles.
\end{definition}

This is a fairly mild assumption that is satisfied in most realistic free physical theories on compact orientable manifolds (and the theory admits an extension to describe classical observables in interacting theories also, as described in \cite{Book2} and \cite{CostelloSH}).  We describe some free examples (which all admit interacting extensions described by Costello and Gwilliam).
\begin{example} \label{ellipticexamples}
\begin{enumerate}
 \item \textbf{Scalar field theories}\\
 Let $\Phi(U) = C^\infty(U)$, the sheaf of smooth functions on a compact Riemannian manifold, and let $\mc L$ be the Lagrangian density for a free scalar field of mass $m$, namely
 \[\mc L(\phi) = \d\phi \wedge \ast d\phi - m^2 \phi \wedge \ast \phi,\]
 where $\ast$ is the Hodge star.  The smeared classical observables in this theory are generated by the elliptic complex $C^\infty(X)[1] \overset Q \to C^\infty(X)$ where $Q = \Delta - m^2$.
 \item \textbf{Abelian Chern-Simons theory}\\
 Let $T$ be a torus, let $P \to X^3$ be a principal $T$ bundle on a compact 3-manifold, and let $\Phi(U)$ be the sheaf describing connections on $P$ (where we trivialize the torsor by choosing a fixed reference connection).  Chern-Simons theory on this fixed bundle is described by the complex $\Omega^*(U; \mf t)$, where $\mf t$ is the (abelian) Lie algebra of $T$.  Equivalently, it is obtained as the classical theory whose fields are given by the complex $\Omega^0(U; \mf t)[1] \to \Omega^1(U; \mf t)$ with Lagrangian density $\LL_{\mr{CS}}(A) = [A \wedge \d A]$.  Taking a Dolbeault complex on a complex manifold instead of a de Rham complex describes instead \emph{holomorphic} Chern-Simons theory (and indeed, in this language it makes sense to define Chern-Simons theory in any dimension as the theory whose algebra of classical observables are built from this complex).
 \item \textbf{Abelian Yang-Mills theory}\\
 Let $T$ be a torus, let $P \to X^4$ be a principal $T$ bundle on a compact Riemannian 4-manifold, and let $\Phi(U)$ be the sheaf describing connections on $P$ as above.  Yang-Mills theory on this fixed bundle is described by the shifted cotangent to the Atiyah-Singer-Donaldson complex describing anti-self-dual connections on $P$, that is, the complex
 \[\xymatrix{
  \Omega^0(U; \mf t) \ar[r]^\d &\Omega^1(U; \mf t) \ar[r]^{\d_+} &\Omega^2_+(U; \mf t) \\
  &\Omega^2_+(U; \mf t) \ar[r]^\d \ar[ur]^1 &\Omega^3(U; \mf t) \ar[r]^\d &\Omega^4(U; \mf t)
 }\]
 in degrees -1 to 2.  Here $\Omega^2_+(U; \mf t)$ denotes the space of self-dual 2-forms and $\d_+$ is the composition of $\d$ with projection onto this space.  
\end{enumerate}
\end{example}
We'll see shortly examples of \emph{$p$-form} theories that also fit into this framework.

\begin{definition}
The \emph{quantum} smeared observables on an open set $U$ are given by the cochain complex with the same underlying graded abelian group as the classical smeared observables, but with differential $\sym(Q^\text{sm}) + D$, where $D$ is the \emph{smeared BV operator} extended from the operator $\sym^2(\mc E(U)) \to \sym^0(\mc E(U))$ given by the invariant pairing restricted to $\mc E(U)$ according to the BD product formula 
\[D(\phi \cdot \psi) = D(\phi) \cdot \psi + (-1)^{|\phi|} \phi \cdot D(\psi) + \{\phi, \psi\}.\]
\end{definition}

\begin{remark}
The quantum smeared observables embed into the whole complex of quantum observables as a subcomplex, dense in each degree.  To see this one just needs to check that the quantum BV operators commute with the inclusion, which follows directly from the definitions.
\end{remark}

\section{Generalized Maxwell Theories as Prefactorization Algebras} \label{sectionGenMaxwell}
Having described the general formalism, we'll explain the specific theories which we'll be studying: the \emph{generalized Maxwell theories}, a family of theories including as its simplest two examples sigma models with target a torus and abelian pure Yang-Mills theories.

\subsection{Generalized Maxwell Theories}
\subsubsection{The Classical Prefactorization Algebra}
The construction of the previous chapter yields a classical prefactorization algebra of observables from the above data.  We can describe the local sections fairly concretely.  

\begin{lemma}
The local observables $\obscl_{\pmp}(U)$ associated to the generalized Maxwell theory on an open set $U$ can be identified as the product over $\mr H^p(U;2\pi R \ZZ)$ of the free cdga generated by the complex
\[\Omega_c^{\le p-1}(U; \CC)[p-1] \overset Q \to \Omega^{\le p-1}(U; \CC)^\vee[1-p],\]
where $Q$ is the linear operator sending a compactly supported $(p-1)$-form $\alpha$ to the linear functional $\beta \mapsto \int \d \beta \wedge \ast \d \alpha$ on the space of $(p-1)$-forms.
\end{lemma}

\begin{proof}
We start with the algebra of polyvector fields
\[\PV(\pmp(U)) = \sym(V(U)) \otimes \OO(\pmp(U)),\]
where $V(U) = \Omega^{\le p-1}(U; \CC)[p-1]$ is the shifted truncated de Rham complex.  We must modify this algebra by incorporating the classical differential coming from the action functional (or rather, its first variation $\d S$).  This is a cochain map 
\[\Omega^{\le p-1}(U; \CC)[p-1] \to \OO(\pmp(U)),\]
i.e. a linear map from $\Omega^{p-1}(U;\CC)$ that vanishes on exact $(p-1)$-forms.  We can define such an operator using the curvature map $F \colon \pmp(U) \to \Omega^p_{\cl}(U)$.  This induces a pullback map $F^* \colon \OO(\Omega^p_{cl}(U)) \to \OO(\pmp(U))$.  By composing with the curvature map it suffices to define the classical differential as the map $\Omega^{p-1}_c(U; \CC) \to \OO(\Omega^p_{cl}(U; \CC)) \overset {F^*} \to \OO(\pmp(U))$
\[\alpha \mapsto \left( \beta \mapsto \int_U \beta \wedge \ast \d\alpha \right) \mapsto \left( A \mapsto \int_U F_A \wedge \ast \d\alpha \right) = \iota_{\d S}(\alpha).\]
This functional, sending a field $A$ to $\int_U F_A \wedge \ast \d\alpha$, recovers the first variation of the required action functional.  On each sector, associated to a class in $\mr H^p(U; 2\pi R\ZZ)$, this map is just the map $Q$ sending $\alpha$ to the linear functional $\beta \mapsto \int \d \beta \wedge \ast \d \alpha$ on the space of $(p-1)$-forms as required.
\end{proof}

\subsubsection{The Quantum Prefactorization Algebra}
Now, we know abstractly how to quantize this prefactorization algebra, but we should see what it actually means in this context.  The following statement is an unwinding of the definitions of Section \ref{quantization_section} for the example of generalized Maxwell theories.

\begin{lemma} \label{quantum_pmp_lemma}
The local quantum observables $\obsq_{\pmp}(U)$ associated to the generalized Maxwell theory on an open set $U$ can be identified as the deformation of $\obscl_{\pmp}(U)$ by the BV operator $D$ generated, according to the BD operator formula \ref{BD_formula}, by an evaluation map $\ev \colon \Omega^{p-1}_c(U; \CC) \otimes \OO(\pmp(U)) \to \OO(\pmp(U))$ (\ref{evaluation_eqn}) defined as follows.
\begin{enumerate}
 \item Start with an element $\chi \otimes f \in \Omega^{p-1}_c(U; \CC) \otimes \OO(\pmp(U))$.  We first take the exterior derivative of both $\chi$ and $f$ to yield
 \[\d\chi \otimes \d f \in \Omega^{p}_{c,cl}(U; \CC) \otimes \Omega^1(\pmp(U)) \iso \Omega^{p}_{c,cl}(U; \CC) \otimes \OO(\pmp(U)) \otimes \Omega^p_{cl}(U; \CC)^\vee.\]
 \item Use the evaluation pairing between $\Omega^p_{cl}(U; \CC)^\vee$ and $\Omega^{p}_{c,cl}(U; \CC)$ (that is, between linear vector fields and linear 1-forms on the fields) to produce a contracted element $\d f(\d\chi) \in \OO(\pmp(U))$ as required.
\end{enumerate}
\end{lemma}

From now on, when we write $\obsq(U)$ we'll be referring specifically to the quantum observables in a generalized Maxwell theory (which one will generally be clear from context).  

We'll write $\obsq(U)_0$ to refer specifically to the \emph{gauge invariant degree zero} observables: the part of the cochain complex that refers to actual observables in the usual sense of the word (that is, functionals on the physical, i.e. degree 0, fields), as opposed to encoding relationships between observables.  The notation doesn't refer to the entire degree zero part of the cochain complex, but rather to the following subcomplex.

\begin{definition} \label{deg_0_observables}
Write $\obsq(U)_0$ to indicate the complex $\OO(\mr H^0(\Phi(U))) \sub \obsq_\Phi(U)$ equipped with the quantum BV differential, where we're using the fact that the projection $\Phi(U) \to \mr H^0(\Phi(U))$ induces a pullback map $\OO(\mr H^0\Phi(U))) \to \OO(\Phi(U))$.
\end{definition}

\begin{remark}
This notion is only well-defined because our quantum field theories are free; in the presence of an interaction there is an additional differential in the complex of local observables which needn't descend to a well-defined operator on $\mr H^0(\Phi(U))$.
\end{remark}

\subsection{Free Theories from $p$-forms}
In order to do calculations with Maxwell theories we will relate them to simpler free field theories where the fields are sheaves of $p$-forms.  This will correspond, intuitively, to considering observables that factor through the curvature map $\widehat {\mr H}^p(U) \to \Omega^p_{cl}(U)$.  These theories will be especially straighforward in the sense that the action functional will involve no derivatives at all, so the classical BV operator will be a differential operator of degree zero.

\begin{definition}
Fix $0 < p < n$ as before.  The \emph{free $p$-form theory} on $X$ with coupling constant $R$ is the Lagrangian field theory with sheaf of fields given by the sheaf of vector spaces $\Omega^p$ and action functional
\[S_R(\alpha) = R^2\|\alpha\|^2_2 = R^2 \int_X \alpha \wedge \ast \alpha\]
given by the $L^2$-norm.  The \emph{free closed $p$-form theory} with coupling constant $R$ is the theory with sheaf of fields $\Omega^{\ge p}[p]$ -- a resolution of the sheaf of closed $p$-forms, with the same action functional.
\end{definition}

We can build classical and quantum factorization algebras directly from this data as a very easy application of the BV formalism described above: we'll denote them by $\obscl_{\Omega^p}$, $\obsq_{\Omega^p}$ etc, with the choice of $R$ surpressed.  For the general $p$-form theory we compute 
\[\obscl_{\Omega^p}(U) = \sym(\Omega^p_c(U)[1] \overset {R^2 \iota} \to \Omega^p(U)^\vee)\]
where $\iota$ is the inclusion of $\Omega^p_c(U)$ into the dual space given by the $L^2$-pairing.  This follows directly from the construction of the action functional given in Section \ref{sectionAction}: this operator describes the first variation of the action functional. 

The quantum BV operator $D$ is induced from the evaluation pairing $\Omega^p_c(U) \otimes \Omega^p(U)^\vee \to \RR$.  We can produce a smeared complex of quantum observables using the standard $L^2$ pairing on $p$-forms coming from the Riemannian metric on $X$.  That is, we have local smeared quantum observables given by
\[\obssm_{\Omega^p}(U) = (\sym(\Omega^p_c(U)[1] \oplus \Omega^p_c(U)), \sym(\cdot R^2) + D)\]
where $\cdot R^2$ is now just a scalar multiplication operator, and $D$ is the operator extended from the $L^2$ pairing as a map $\sym^2 \to \sym^0$ according to the usual BD formula.  This complex is quasi-isomorphic to $\RR$ for every $U$.  Indeed, by a spectral sequence argument it suffices to check this for the classical observables, where we're computing $\sym$ of a contractible complex.

\subsubsection{The Closed $p$-Form Theory}
The closed $p$-form theory is similar until we smear.  That is, the quantum factorization algebra is given as the complex
\[\obsq_{\Omega^p_{cl}}(U) = (\sym(\Omega^{\ge p}_{c}(U)[p+1] \oplus \Omega^{\ge p}(U)^\vee[-p]), \sym(R^2 \iota) + D)\]
where $D$ is induced from the evaluation pairing as above.  The theory is not quite the same after smearing, in particular it is no longer locally contractible.  

\begin{lemma}
There is a smeared version $\obssm_{\Omega^p_{cl}}(U)$ of the classical factorization algebra for the closed $p$-form theory, given by the free algebra generated by the cochain complex
\begin{equation} \label{smeared_closed}
\xymatrix{
&&\Omega^p_c(U)[1] \ar[rd]^{2\pi R} \ar[r]^{\d} & \Omega^{p+1}_c(U) \ar[r]^{\d} &\cdots \ar[r]^(.25){\d} &\Omega^n_c(U)[p-n+1] \\
\Omega^n_c(U)[n-p] \ar[r]^(.75){\d^*} &\cdots \ar[r]^(.4){\d^*} &\Omega^{p+1}_c(U)[1] \ar[r]^{\d^*} &\Omega^p_c(U),
}\end{equation}
where $\d^*$ is the adjoint to the de Rham differential $\d$ associated to the metric, deformed by the BV operator $D$ is induced by the $L^2$-pairing. 
\end{lemma}

\begin{proof}
To see that this is a well-defined smearing of the classical observables we only need to note that the triangle of complexes
\[\xymatrix{
\Big( \Omega^p_c(U) \ar[rrd]^{2\pi R} \ar[r]^\d &\cdots \ar[r]^\d   &\Omega^n_c(U) \Big) \ar@<10pt>[rrd] &&&\\
\Big( \Omega^n_c(U) \ar[r]^{\d^*} &\cdots \ar[r]^{\d^*} &\Omega^p_c(U) \Big) \ar[r] &\Big(\Omega^p(U) \ar[r]^\d &\cdots \ar[r]^\d &\Omega^n(U) \Big)^\vee
}\]
commutes (the arrow from top to bottom-right is the classical operator $Q$ and the arrow from top to bottom-left is the smeared classical operator $Q^{\text{sm}}$).  Then, to ensure that the smeared BV operator is well-defined, we use the fact that local closed and local coexact forms are orthogonal with respect to the $L^2$-pairing.
\end{proof}

This story proceeds identically for complex-valued forms, which we'll use from now on.  The complexification of the closed $p$-form theory is -- by design -- closely related to the generalized Maxwell theory of order $p$: the complexes of observables are isomorphic on open sets $U$ where $H^p(U;\ZZ) = 0$, and one has a natural map on \emph{degree zero} observables $\obsq_{\Omega^p_{cl}}(U; \CC)_0 \to \obsq_{\pmp}(U)_0$ for any $U$.  This map should be thought of as the inclusion of those observables that factor through the curvature map. 

\begin{lemma} \label{closed_Maxwell_relation_lemma}
There is a morphism $\obsq_{\Omega^p_{\cl}} \to \obsq_{\pmp}$ of prefactorization algebras induced by the curvature map 
\[F\colon \ZZ(p)_{\mc D} \to \Omega^{\ge p}[p]\]
of sheaves of cochain complexes.  This map is a quasi-isomorphism on contractible open sets (but not generally, since $\obsq_{\pmp}$ is only a prefactorization algebra).
\end{lemma}

\begin{proof}
We first extend the degree zero map to a map on all classical observables by pulling back polyvector fields.  Concretely this is the map
\[F \colon \sym(\Omega^{\ge p}_{c}(U; \CC)[p+1]) \otimes \OO(\Omega^p_{cl}(U; \CC)[p]) \to \sym(\Omega^{\le p-1}_c(U;\CC)[p]) \otimes \OO(\Phi(U))\]
given by $F^*$ in the second factor and by the formal adjoint $\d^* \colon \Omega^p_c(U;\CC) \to \Omega^{p-1}_c(U;\CC)$ to the de Rham differential in the second factor.  One needs to check that this is compatible with the classical BV operator $Q$, which requires observing that the square
\[\xymatrix{
 \Omega^{\ge p}_{c}(U; \CC)[p+1] \ar[r] \ar[d]_{\d^*} &\OO(\Omega^{\ge p}(U; \CC)[p]) \ar[d]^{F^*} \\
 \Omega^{\le p-1}_c(U;\CC)[p]  \ar[r]  &\OO(\pmp(U)) 
}\]
commutes, where the horizontal arrows are those maps defining the Poisson brackets.  If, further, the open set $U$ is contractible then this map defines a \emph{quasi-isomorphism} of the complex of local classical observables.

We then need to check that our map commutes with the quantum BV operator; that is, we check that the triangle
\[\xymatrix{
\Omega^{\ge p}_{c}(U; \CC) \otimes (\Omega^{\ge p}(U; \CC))^\vee \ar[r] \ar[d]_{\d^* \otimes F^*} &\CC \\
\Omega^{\le p-1}(U;\CC) \otimes \OO(\pmp(U)) \ar[ur]^D
}\]
commutes, where $D$ is the quantum BV operator in the generalized Yang-Mills theory.  This is clear from the definition of the map $D$ as in Lemma \ref{quantum_pmp_lemma}, which first applies the exterior derivative to the linear vector field and the linear functional, then pairs the result via the $L^2$-pairing.  Again, if $U$ is in fact contractible then the resulting map on local observables is actually a quasi-isomorphism.
\end{proof}

\section{Expectation Values} \label{sectionExpVal}
In this section we'll explain how to compute vacuum expectation values of observables in free theories where the fields are given by a cochain complex of vector spaces, such as the closed $p$-form theories introduced in the previous section.  While the method does not apply directly to free theories with a more complicated space of fields, we show that the expectation value can be computed by functional integrals, which \emph{will} generalize to more complicated settings like that of generalized Maxwell theories.

\subsection{Expectation Values from Free Quantum Prefactorization Algebras}
For an elliptic theory $(\Phi, \mc L)$, consider the complex of global classical observables $\obs(X)$ with underlying graded vector space $\OO(T^*[-1]\Phi)$.  This cdga is freely generated by a cochain complex, and its smeared version is assumed to be freely generated by an elliptic complex $\mc E$ as described in Section \ref{smearing}.  Suppose that $X$ is compact.  Now, Hodge theory gives us a Laplacian operator $\Delta \colon \mc E \to \mc E$ and a splitting in each degree: $\mc E_i = \mc H_i \oplus \mc H_i^\perp$, where $\mc H_i$ denotes the finite-dimensional vector space of \emph{harmonic} elements in degree $i$.  In particular, we can apply this to the degree zero elements $\mc E_0$, which represent a linearization of the space of the (degree zero) \emph{fields} in our theory.

\begin{hypothesis}
From now on, $X$ will always denote a compact Riemannian manifold without boundary.  In particular, Hodge theory tells us that if $\mc E$ is the sheaf of sections of an elliptic complex of vector bundles on $X$, then the complex of global sections of $\mc E$ has finite-dimensional cohomology.
\end{hypothesis}

\begin{definition}
The Hodge decomposition defines a splitting
\[\mc E_0 = \mc M \oplus \mc M^\perp\]
where $\mc M$ (or equivalently, $\mc H_0$ in the notation above) denotes the finite-dimensional space of harmonic fields.  We call $\mc M$ the space of \emph{massless modes} of the theory, and $\mc M^\perp$ the space of \emph{massive modes}.
\end{definition}

The terminology comes from the example of the free scalar field.
\begin{example}
Let $\Phi(U) = C^\infty(U)$ with the action for a free scalar field of mass $m \ge 0$, i.e.
\[S(\phi) = \int_X \left(\phi \Delta \phi - m^2 |\phi|^2\right) \dvol.\]
The complex of smeared quantum global observables here, as remarked upon in \ref{ellipticexamples} has underlying graded vector space $\sym(C^\infty(X)[1] \oplus C^\infty(X))$, classical differential $\sym(\Delta - m^2)$ and quantum differential induced from the $L^2$-pairing on functions.  The relevant elliptic complex then is the two-step complex
\[C^\infty(X) \overset {\Delta-m^2} \to C^\infty(X)\]
in degrees $-1$ and 0.  The cohomology of this complex is finite-dimensional since $X$ is assumed to be compact, so the $m^2$-eigenspace of the Laplacian is finite-dimensional (or just by Hodge theory, since the complex is elliptic). 

Focusing on the case $m=0$, the Hodge decomposition splits $C^\infty(X)$ in each degree as a sum of eigenspaces for $\Delta^2$, or equivalently as a sum of eigenspaces for $\Delta$.  The eigenspaces are the spaces of solutions to $\Delta \phi = \lambda \phi$, which we think of as the energy $\sqrt \lambda$ pieces of the space of fields (by analogy with the Klein-Gordon equation in Lorentzian signature).  The cohomology is then represented by the harmonic / massless piece.
\end{example}

\begin{lemma}
Suppose $(\Phi, \LL)$ is an elliptic classical field theory on a manifold $X$ as in Definition \ref{elliptic_theory_def}, and suppose that the zeroth cohomology $\mr H^0(\mc E)$ of the elliptic complex $\mc E$ is trivial (for instance, if $X$ is compact as above, and the theory has no massless modes).  Then the space of global degree 0 quantum smeared observables $\obssm_\Phi(X)_0$ (as in Definition \ref{deg_0_observables}) is quasi-isomorphic to $\CC$.
\end{lemma}

\begin{proof}
By definition, the smeared \emph{classical} degree 0 observables on $X$ are given by the symmetric algebra $\sym(\mr H^0(\mc E))$ (note: not the larger object $\mr H^0(\sym(\mc E))$), which is isomorphic to $\sym(0) = \CC$ in degree zero by assumption. To see this is also true for the smeared \emph{quantum} degree 0 observables we use a simple spectral sequence argument, using the filtration of the complex by $\sym$ degree.  The BV operator is extended from the map from $\sym^2$ to $\sym^0$ by the $L^2$ pairing, so in general lowers $\sym$-degree by two.  The $E_1$ page of the spectral sequence computes the cohomology of the classical complex of smeared observables (i.e. the cohomology with respect to only the $\sym$ degree 0 part of the differential), and the spectral sequence converges to the cohomology of the complex of smooth quantum observables (i.e. the cohomology with respect to the entire differential).  Since the $E_1$ page is quasi-isomorphic to $\CC$ in degree 0, so must be the $E_\infty$ page.
\end{proof}

\begin{remark} \label{weaker_remark}
In fact it's enough to assume a weaker condition: instead of requiring that $\mr H^0(\mc E)$ vanishes, it's enough to assume that every degree zero element of the factor $\Phi_c$ of $\mc E$ coming from ordinary fields, rather than BV antifields, vanishes in cohomology.
\end{remark}

Note that in these circumstances there is a \emph{canonical} quasi-isomorphism from the global smeared observables to $\CC$, characterized by the property that the class of the observable 1 in $\sym^0$ maps to 1.

\begin{definition}
The \emph{expectation value} in an elliptic theory with trivial cohomology is the canonical isomorphism
\[(\obssm_\Phi(X))_0 \to \CC\]
sending $[1]$ to 1.
\end{definition}

\begin{hypothesis}
From now on we'll always suppose that the complex $\mc E$ satisfies the condition of Remark \ref{weaker_remark}, so that we can define the expectation value of a degree 0 observable.
\end{hypothesis}

\begin{remark}
With the setup we've been using, on compact $X$ the idea that the theory had no massless modes was essential.  The massless modes correspond to the locus in $\mc E_0$ where the action functional vanishes, so the locus where the exponentiated action is \emph{degenerate}.  Since we'll be computing expectation values as a limit of finite dimensional Gaussian integrals it will be important to ensure that there are no massless modes so that the Gaussian is non-degenerate, and so the finite dimensional Gaussian integrals give finite answers.

It might be possible, in a somewhat different formalism, to work with a non-linear space of fields splitting into a \emph{linear} space of massive modes and a \emph{compact finite-dimensional} moduli space of massless modes.  One could then describe an expectation value by integrating out the space of massive modes over each point in the moduli space of massless modes (using the formalism we will describe below) to produce a section of a rank 1 local system.  If this local system was actually trivialisable then such a section could be integrated to give a number.  Failure of trivialisability would be an example of an \emph{anomaly} for a free field theory.
\end{remark}

\begin{example}
The theory of $p$-forms has smeared complex $\mc E = \Omega_c^p \overset {R^2} \to \Omega_c^p$, which is contractible.  This theory, therefore, has no massless modes.
\end{example}

\begin{example}
The theory of closed $p$-forms has smeared complex as in \ref{smeared_closed}.  This complex \emph{does} have cohomology in degree 0, but all non-zero degree 0 cohomology classes are represented by a cocycle of the form $(0, \alpha) \in \Omega^{p+1}_c(U) \oplus \Omega^{p}_c(U)$, and therefore this theory satisfies the weaker condition of Remark \ref{weaker_remark}.
\end{example}

\subsection{Computing Expectation Values}
The ideas of this section are not original, but are merely a recollection of physical techniques, whose analysis is well understood, in the present context. A modern mathematical account of the relationship between functional integrals, Feynman diagrams and homological algebra can be found in \cite{GwilliamJohnsonFreyd}.

\subsubsection{Feynman Diagrams for Free Theories}
So, let's fix a free elliptic theory with fields $\Phi$ and no massless modes: for instance a $p$-form or closed $p$-form theory on a compact manifold $X$.  The idea of the Feynman diagram expansion is to compute expectation values of observables in our theory combinatorially.  The crucial idea that we'll use in order to check that we can do this is that --for smeared observables -- the expectation value map is \emph{uniquely characterized}.  That is, for smeared observables there is a unique quasi-isomorphism from global smeared observables to $\CC$ that sends 1 to 1.  Therefore to check that a procedure for computing expectation values is valid it suffices to check that it is a non-trivial quasi-isomorphism, then rescale so the map is appropriately normalized.

Take a global degree zero smeared observable $\OO \in \sym(\Phi(X)_0)$ which is \emph{gauge invariant}. That is, we consider observables that can be written as a product of linear observables
\[\OO = \OO_1^{n_1}\OO_2^{n_2} \cdots \OO_k^{n_k}\]
where $\OO_1, \ldots \OO_k$ are linearly independent linear smeared observables in $\Phi(X)_0$ (not necessarily gauge invariant themselves), and where $\OO$ is \emph{closed} in the classical (or equivalently quantum) complex of smeared observables. This corresponds to gauge invariance because the only differential \emph{from} the degree zero observables in either the classical or quantum complexes of smeared observables comes from $\sym$ of the underlying differential on $\Phi$.  Closed elements are observables in the kernel of $\sym(d^\vee)$ where $\d \colon \Phi_{-1} \to \Phi_0$ is the underlying differential, and where the adjoint $\d^\vee$ is defined on $\Phi_0$ using the invariant pairing on fields.  The kernel of $\d^\vee$ corresponds exactly to the cokernel of $\d$, which is the space of degree zero observables we think of as invariant under gauge transformations.  

Generally gauge invariant polynomial observables are sums of monomial observables of this form, and we extend the procedure of computing duals linearly, so it suffices to consider $\OO$ of this form.

We compute the expectation value of $\OO$ combinatorially as follows.  Depict $\OO$ as a graph with $k$ vertices, and with $n_i$ half edges attached to vertex $i$.  The expectation value $\langle \OO \rangle$ of $\OO$ is computed as a sum of terms constructed by gluing edges onto this frame in a prescribed way.  Specifically, we attach \emph{propagator edges} -- which connect together two of these half-edges -- in order to leave no free half-edges remaining.  A propagator between linear observables $\OO_i$ and $\OO_j$ receives weight via the pairing 
\[\frac 12 \int_X \langle\OO_i, Q^{-1} \OO_j\rangle\]
where $Q$ is the classical BV operator, $Q^{-1}$ is defined by inverting $Q$ on each eigenspace for the corresponding Laplacian (using the non-existence of massless modes), and $\langle -,-\rangle$ is the invariant pairing on smeared observables.  A diagram is weighted by the product of all these edge weights.  The expectation value is the sum of these weights over all such diagrams.
\begin{figure}[ht]
 \centering
 \includegraphics[height=3.5cm]{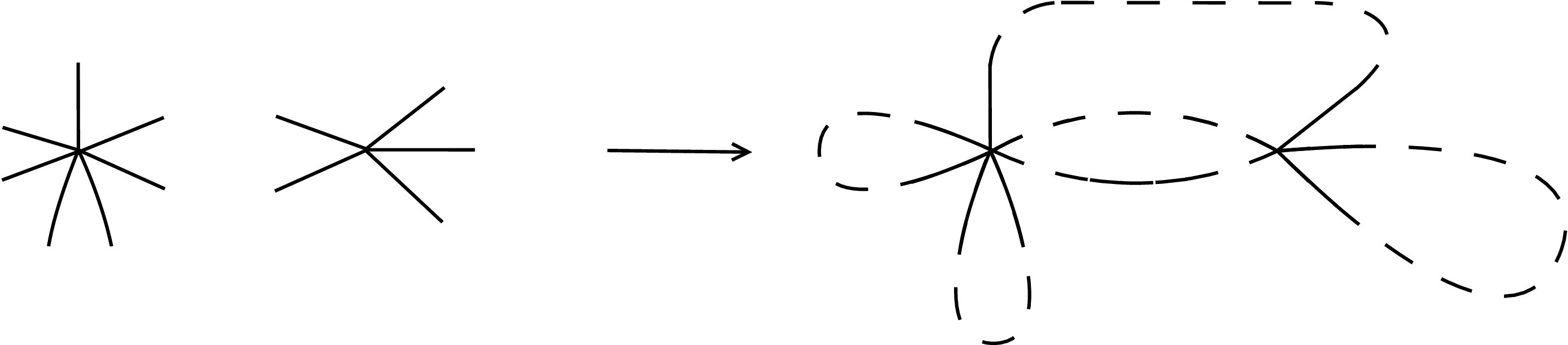}
 \caption{One of the terms in the Feynman diagram expansion computing the expectation value of an observable of form $\OO_1^7\,\OO_2^5$.  On the left we see the starting point, with half-edges, and on the right we see one way of connecting these half-edges with propagator edges (indicated by dashed lines).}
\end{figure}

To check that this computes the expectation value, we must show that it is non-zero, and that it vanishes on the image of the differential in the complex of quantum observables.  The former is easy: the observable 1 has expectation value 1 (so we're also already appropriately normalized).  For the latter, we'll show that the path integral computation for degree zero global observables in $\obsq(U)_0$ arise as a limit of finite-dimensional Gaussian integrals, and that the images of the quantum BV differential are all divergences, so vanish by Stokes' theorem.

\subsubsection{Regularization and the Path Integral}
The classical complex of linear observables in our theory is elliptic, so induces a Laplacian operator $\Delta$ acting on $\mc E_0$ with discrete spectrum $0 < \lambda_1 < \lambda_2 < \cdots$ and finite-dimensional eigenspaces.  Let $F_k\mr H^0(\Phi(X))$ denote the sum of the first $k$ eigenspaces: this defines a filtration of the global degree zero (linearized) fields by finite-dimensional subspaces.  We recall a standard result about infinite dimensional Gaussian integrals.

\begin{prop} \label{expvals_gaussian}
Let $\OO$ be a smeared global observable.  The finite-dimensional Gaussian integrals
\[\frac 1{Z_k} \int_{F^k\mr H^0(\Phi(X))} \OO(a) e^{-S(a)} da,\]
where $Z_k$ is the volume $\int_{F^k\mr H^0(\Phi(X))} e^{-S(a)} da$, converge to a real number $I(\OO)$ as $k \to \infty$, and this number agrees with the expectation value computed by the Feynman diagrammatic method.
\end{prop}

\begin{proof}
We check that for each $k$ the Gaussian integral admits a diagrammatic description, and observe that the expressions computed by these diagrams converge to the expression we want.  We may assume as usual that $\OO$ splits as a product of linear smeared observables $\OO = \OO_1^{n_1} \OO_2^{n_2} \cdots \OO_\ell^{n_\ell}$.  The $\OO_i$ describe linear operators on the filtered pieces.  We can write the Gaussian integral using a generating function as
\[\int_{F^k\mr H^0(\Phi(X))} \OO(a) e^{-S(a)} da = \left. \frac{\dd^{n_1+ \cdots + n_\ell}}{\dd t_1^{n_1} \cdots \dd t_\ell^{n_\ell}} \right|_{t_1 = \cdots = t_\ell = 0} \int_{F^k\mr H^0(\Phi(X))} e^{-\int_X\langle a, Qa \rangle + t_1\OO_1(a)  \cdots + t_\ell \OO_\ell(a)} da,\]
provided that $k$ is large enough that upon projecting to $F^k\mr H^0(\Phi(X))$ the $\OO_i$ are linearly independent.  Call this projection $\OO_i^{(k)}$.  This expression is further simplified by completing the square, yielding
\[Z_k \left. \frac{\dd^{n_1+ \cdots +n_\ell}}{\dd t_1^{n_1} \cdots \dd t_\ell^{n_\ell}} \right|_{t_1 = \cdots = t_\ell = 0} e^{\frac 12 \int_X \langle (t_1\OO_1^{(k)} + \cdots t_\ell\OO_\ell^{(k)}), Q^{-1} (t_1\OO_1^{(k)} + \cdots t_\ell\OO_\ell^{(k)})\rangle }\]
where we've identified the linear smeared observables with differential forms.  We can now compute the Gaussian integral diagrammatically. The $t_1^{n_1}\cdots t_\ell^{n_\ell}$-term of the generating function is the sum over Feynman diagrams as described above, where a diagram is weighted by a product of matrix elements $\frac 12 \int_X \langle \OO_i^{(k)}, Q^{-1}\OO_j^{(k)} \rangle$ corresponding to the edges. We see that as $k \to \infty$ this agrees with the weight we expect.
\end{proof}

Now we can justify why the expectation value vanishes on the image of the quantum differential.  Let $\OO \in \obssm_{\Phi}(X)_0$ be a smeared degree 0 global observable, and suppose $\OO = d_\Phi V + (D - \iota_{\d S})W$ is in the image of the quantum differential.  The exact term $\d_\Phi V$ is zero in $\mr H^0$ of the fields, so it suffices to consider the $W$ piece.  The restriction of $W$ to a filtered piece is a vector field on the vector space $F^k\mr H^0(\Phi(X))$, and we can compute the divergence
\[\text{div}(e^{-S(a)} W) = \OO e^{-S(a)},\]
where the restriction to the filtered piece is suppressed in the notation.  So the expectation value of $\OO$ is a limit of integrals of divergences, which vanish by Stokes' theorem, and the expectation value is zero.  This implies that the procedure described above does indeed compute the cohomology class of a global smeared observable in the canonically trivialized cohomology.

\section{Fourier Duality for Polynomial Observables} \label{sectionFourier}
In its simplest form, Fourier duality is an isomorphism on degree 0 observables in the free $p$-form theories: $\obsq_{\Omega^p,R}(U)_0 \iso \obsq_{\Omega^{n-p},1/2R}(U)_0$.  It will not extend to any kind of cochain maps in these theories, and in particular will not be compatible with the expectation value maps, but we'll show that it \emph{is} compatible with the expectation values after the restriction $\obsq_{\Omega^p,R} \to \obsq_{\Omega^p_{cl},R}$ to the closed $p$-form theories.

\begin{remark}
The reason we refer to the duality as defined here as \emph{Fourier} duality is that, as we'll show in Proposition \ref{fourier_gaussian}, the dual observable to an observable $\OO$ can be computed as the (appropriately regularized) Fourier transform of $\OO$ in the vector space of closed $p$-forms.  While we'll begin with a more axiomatic definition, this Fourier transform should be thought of as motivating the combinatorial definition.
\end{remark}

\subsection{Feynman Diagrams for Fourier Duality}
We'll construct the Fourier transform in an explicit combinatorial way using Feynman diagrams extending the Feynman diagram expression computing expectation values.  Take a smeared monomial observable $\OO \in \obssm_{\Omega^p,R}(U)_0$.  As above, we write $\OO$ as
\[\OO_1^{n_1}\OO_2^{n_2} \cdots \OO_k^{n_k}\]
where $\OO_1, \ldots \OO_k$ are linearly independent linear smeared observables in $\Omega^p_{c}(U)$.

We compute the Fourier dual of $\OO$ in a similar diagrammatic way to the method we used to compute expectation values.  Depict $\OO$ as a graph with $k$ vertices, and with $n_i$ half edges attached to vertex $i$. Now, we can attach any number of \emph{propagator edges} as before, and also any number of \emph{source terms} -- which attach to an initial half-edge and leave a half-edge free -- in such a way as to leave none of the original half-edges unused.  The source terms have the effect of replacing a linear term $\OO_i$ with its Hodge dual $\ast \OO_i$, i.e. the observable obtained by precomposition with the Hodge star operator.  The result is a new observable
\[(\ast \OO_1)^{m_1} (\ast\OO_2)^{m_2} \cdots (\ast\OO_k)^{m_k}\]
where $m_i$ is the number of source edges connected to vertex $i$, now thought of as a degree zero observable in $\obssm_{\Omega^{n-p},1/2R}(U)_0$.  
\begin{definition}
The total Fourier dual observable $\wt \OO$ is the sum of these observables over all such graphs, where each observable is weighted by the product over all edges of the corresponding graph of the following weights.
\vspace{-9pt}
\begin{itemize}
 \item A propagator between linear observables $\OO_i$ and $\OO_j$ receives weight 
 \[\frac 1{2R^2} \langle \OO_i, \OO_j \rangle = \frac 1{2R^2} \int_{X} \OO_i \wedge \ast \OO_j.\] 
 \item A source term attached to a linear observable $\OO$ receives weight $i / 2R^2$.
\end{itemize}
\end{definition}

\begin{figure}[ht]
 \centering
 \includegraphics[height=3.5cm]{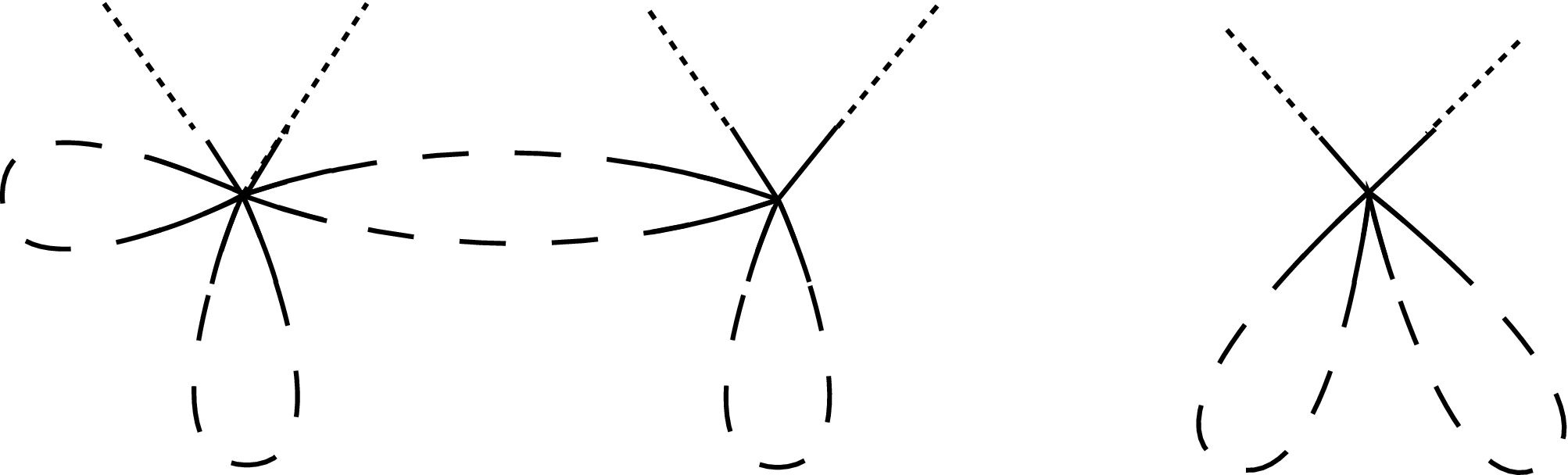}
 \caption{The Feynman diagram corresponding to a degree 6 term in the Fourier dual of an observable of form $\OO_1^8 \,\OO_2^6 \,\OO_3^6$.  Propagators are indicated by dashed lines and sources by dotted lines.}
\end{figure}

From a path integral perspective, these terms have natural interpretations.  The Fourier transform of $\OO$ can be thought of as the expectation value of an observable of form $\OO e^{i \langle a, \wt a \rangle}$ where $a$ is a field and $\wt a$ is its Fourier dual variable.  Alternatively, this can be thought of as a functional derivative of an exponential of form $e^{- S(a) + i \langle a, \wt a \rangle}$.  The propagator terms arise from applying a functional derivative to the action term, while the source terms arise from applying it to the second term, implementing the Fourier dual.  The Hodge star in the source term arises from the specific pairing in the $p$-form theory, namely the $L^2$ pairing $\int a \wedge \ast \wt a$.

\begin{example}
To demonstrate the idea, we compute the Fourier dual observable to $\OO^4$ for $\OO$ a linear smeared observable.  There is one term with no propagator edges and four sources, six with one propagator edge and two sources, and three with two propagator edges and no sources.  The dual is therefore
\[\wt{\OO^4} = \frac 1{16R^8} (\ast \OO)^4 - \frac 6{8R^6}\|\OO\|^2 (\ast \OO)^2 + \frac 3{4R^4} \|\OO\|^4.\]
If $R^2 = 1/2$ and $\|\OO\|=1$ this recovers the fourth Hermite polynomial $\mathrm{He}_4(\ast \OO)$.
\end{example}

We can compute the dual of a general global observable by smearing first, then dualising: the result is that an observable has a uniquely determined smeared dual for each choice of smearing.  In order to compare expectation values of an observable and its dual, the crucial tool that we'll use is Plancherel's formula, which we can rederive in terms of Feynman diagrams.  The first step is to prove a Fourier inversion formula in this language.  In doing so we'll need to remember that after dualising once, the new observable lives in the \emph{dual} theory, with a different action: therefore the weights assigned to edges will be different, corresponding to a different value of the parameter $R$.

We'll also use the convention that the second application of the Fourier transform is the inverse Fourier transform, which assigns weight $-i / 2R^2$ to a source edge, but is otherwise identical.

\begin{prop} \label{doubledual}
A smeared observable $\OO$ is equal to its Fourier double dual $\wt{\wt{\OO}}$.
\end{prop}

\begin{proof}
Let $\OO = \OO_1^{n_1}\OO_2^{n_2} \cdots \OO_k^{n_k}$ as above.  The Fourier double dual of $\OO$ is computed as a sum over diagrams with two kinds of edges: those coming from the first dual and those coming from the second.  We'll show that these diagrams all naturally cancel in pairs apart from the diagram with no propagator edges.  We depict such diagrams with solid edges coming from the first dual, and dotted edges coming from the second dual.
\begin{figure}[ht]
 \centering
 \includegraphics[height=3.5cm]{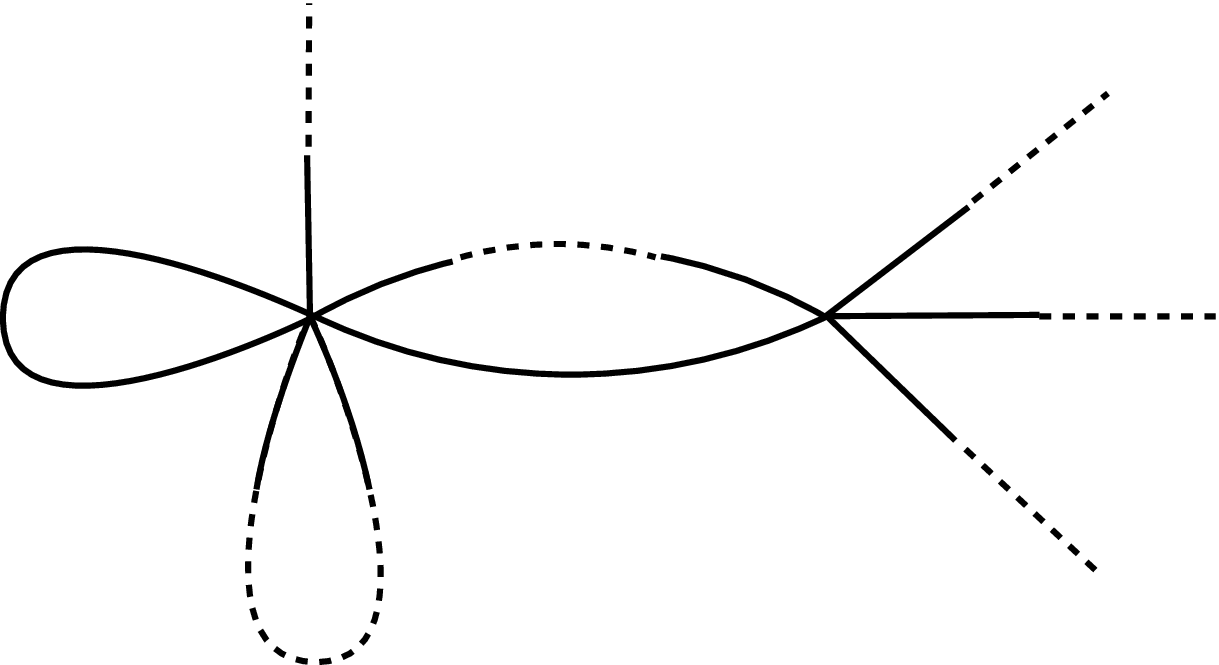}
 \caption{A diagram depicting a summand of the Fourier double dual of an observable of form $\OO_1^7 \, \OO_2^5$.}
\end{figure}

So choose any diagram $D$ with at least one propagator, and choose a propagator edge in the diagram.  We produce a new diagram $D'$ by changing this propagator edge from solid to dotted or from dotted to solid.
\begin{figure}[ht]
 \centering
 \includegraphics[height=3.5cm]{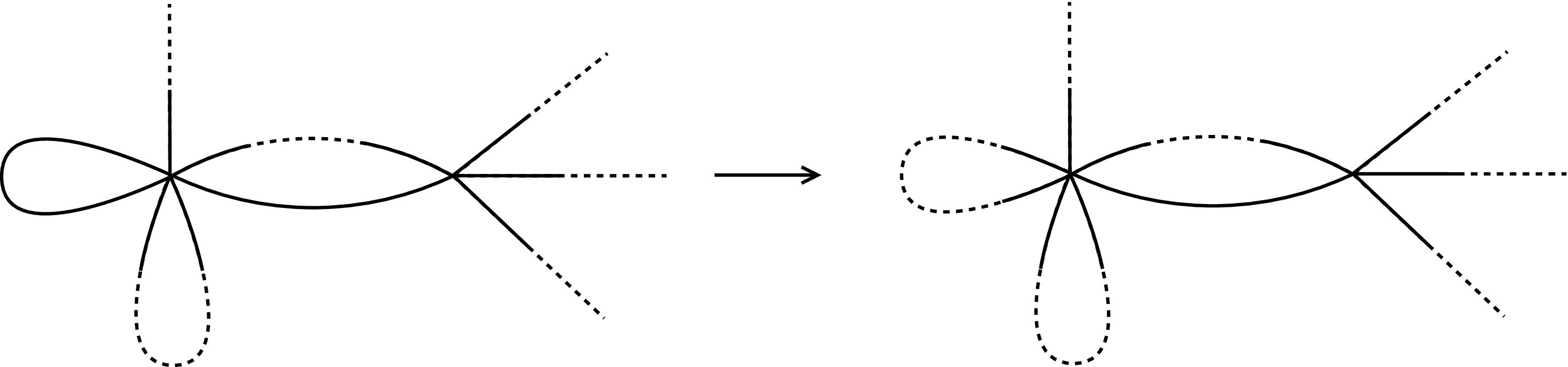}
 \caption{In this diagram we chose the solid leftmost propagator loop (coming from the first dual), and replaced it by two source terms (solid lines) connected with a dotted propagator loop coming from the second dual.}
\end{figure}
 
It suffices to show that the weight attached to this new diagram is $-1$ times the weight attached to the original diagram, so that the two cancel.  This is easy to see: the propagator from the first term contributes a weight $\frac 1{2R^2} \int_{X} \OO_i \wedge \ast \OO_j$.  In the second dual, the weights come from the source terms in the original theory, but the propagator in the \emph{dual} theory, which contributes a weight using the dual theory.  So the total weight is
\[2R^2 \left(\frac i{2R^2}\right)^2  \int_{X} \OO_i \wedge \ast \OO_j\]
Which is $-1$ times the weight of the other diagram, as required.

Finally, we note that the weight assigned to the diagram with no propagator edges in the Fourier double dual is 1.  Indeed, at each free edge, we have a composite of two source terms, contributing a factor of $\left(\frac i {2R^2} \right) \left( \frac{-i}{2(1/2R)^2}\right) = 1$.
\end{proof}

For further justification for these choices of weights, we should compare this combinatorial Fourier duality to one calculated using functional integrals (for global smeared observables).  We'll perform such a check by defining a sequence of Gaussian integrals on filtered pieces, and checking that they converge to a Fourier dual that agrees with the one combinatorially described above.  As before, $F^k\Omega^p(X)$ refers to the filtration by eigenspaces of the Laplacian, this time on the space of all $p$-forms, not just closed $p$-forms.
\begin{prop} \label{fourier_gaussian}
Let $\OO$ be a smeared global observable.  The finite-dimensional Gaussian integrals
\[ \wt \OO_k(\wt a) = \left(\frac 1 {Z_k}\int_{F^k\Omega^p(X)} \OO(a) e^{-S_R(a) + i \int_{X} \wt a \wedge a } da\right) e^{S_{1/2R}(\wt a)}\]
where $\wt a$ is an $(n-p)$-form, converge as $k \to \infty$ to a smeared global observable $\wt O(\wt a) = \lim_{k \to \infty} \wt \OO_k(\wt a)$ which agrees for each $\wt a$ with the Fourier dual observable computed by the Feynman diagrammatic method.
\end{prop}

\begin{proof}
We use the same method of proof as for \ref{expvals_gaussian}, writing the integral as a derivative of a generating function.  Specifically, for $\OO = \OO_1^{n_1} \OO_2^{n_2} \cdots \OO_\ell^{n_\ell}$ we expand
\begin{align*}
 \frac 1 {Z_k}\int_{F^k\Omega^p(X)} \OO(a) e^{-S_R(a) + i \int_{X} \wt a \wedge a} da &= \left. \frac{\dd^{n_1+ \cdots n_\ell}}{\dd t_1^{n_1} \cdots \dd t_\ell^{n_\ell}} \right|_{t_1 = \cdots = t_\ell= 0} \frac 1{Z_k} \int_{F^k\Omega^(X)} e^{S_R(a) + \sum t_i\int_X \OO_i \wedge \ast a+ i \int_{X} \wt a \wedge a} da \\
 &= \left. \frac{\dd^{n_1+ \cdots n_\ell}}{\dd t_1^{n_1} \cdots \dd t_\ell^{n_\ell}} \right|_{t_1 = \cdots = t_\ell = 0} \\
 &\qquad   e^{-S_{1/2R}(\wt a)} e^{\frac 1{4R^2} \int_X(t_1\OO_1^{(k)} + \cdots + t_\ell \OO_\ell^{(k)}) \wedge \ast (t_1\OO_1^{(k)} + \cdots + t_\ell \OO_\ell^{(k)}) + \frac i{2R^2} \int_X (t_1\OO_1^{(k)} + \cdots + t_\ell \OO_\ell^{(k)}) \wedge \wt a} 
\end{align*}
(by completing the square) and extract the $t_1^{n_1} \cdots t_\ell^{n_\ell}$-term.  Once again we're denoting by $\OO_i^{(k)}$ the projection of $\OO_i$ onto the $k^\text{th}$ filtered piece $F^k\Omega^p(X)$.  We choose the level in the filtration large enough so that the upon projecting to the filtered piece the forms $\OO_i$ are linearly independent.  One then observes that in the limit as $k \to \infty$ the relevant term is given by a sum over diagrams as described with the correct weights.
\end{proof}

Now, for any open set $U \sub X$ we have a restriction map of degree zero local observables $r(U) \colon \obssm_{\Omega^p}(U)_0 \to \obssm_{\Omega^p_{cl}}(U)_0$  induced by the projection $\Omega^p_c(U) \to \Omega^p_c(U)/d^*\Omega^{p+1}_c(U)$.  This gives us a candidate notion of duality in the closed $p$-form theory.  So, we might take a degree 0 observable in the image of $r(U)$, choose a preimage, compute the dual then restrict once more.  Of course, this is not quite canonical, because the map $r(U)$ is not injective: the resulting dual observable might depend on the choice of preimage we made.  However, in certain circumstances we might be able to choose a consistent scheme for choosing such a preimage, therefore a canonical duality map.  We'll give such an example in Section \ref{sectionWilsontHooft}, but first we'll prove that for \emph{any} choice of lift, the resulting dual observable in the $\Omega^p_{cl}$ theory has the same expectation value as the original theory.

\begin{remark}
We can also consider duality for closed $p$-form theories with coefficients in a vector space $V$, and -- as we'll observe shortly -- generalized Maxwell theories with gauge group a higher rank torus $T = V / L$, as mentioned in Section \ref{sectionGenMaxwell}.  The theory generalizes in a natural way, with a $p$-form theory with gauge group $T$ dual to an $(n-p)$-form theory with gauge group $\hat T$, the dual torus.  Indeed, there is an identical relationship between the generalized Maxwell theory with gauge group a higher rank torus and a closed $p$-form theory where the forms have coefficients in a vector bundle, and where the classical BV operator is given by the matrix describing the lattice $L$.  By diagonalising this matrix the system separates into a sum of rank one theories, with monomial smeared observables likewise splitting into products of monomial observables in rank one theories which one can dualize individually. 
\end{remark} 

\subsection{Fourier Duality and Expectation Values} \label{Fourierexpvals}
At this point we have two equivalent ways of thinking about both the Fourier transform and the expectation value map for smeared observables: by Feynman diagrams (which allowed us to describe the dual locally) and by functional integration (which allow us to perform calculations, but only globally).  We'll compare the expectation values of dual observables using a functional integral calculation, in which the restriction to \emph{closed} $p$-forms will be crucial.

For an actual equality of expectation values as described above we'll have to restrict to observables on a contractible open set $U$.  Recall this is the setting where the local observables in the closed $p$-form theory agree with observables in the original generalized Maxwell theory.  On more general open sets connections on higher circle bundles are related to only those closed $p$-forms with \emph{integral periods}.  Let's recall this idea, which we described in the context of Deligne cohomology in Remark \ref{Deligne_hypercohomology}.

Recall from Definition \ref{integral_periods_def} that $\Omega^p_{\cl, \ZZ}(X) \sub \Omega^p(X)$ denote the space of closed $p$-forms on a manifold $X$ with \emph{integral periods}, i.e. $p$-forms $\alpha$ where the cohomology class $[\alpha]$ lies in the lattice $\mr H^p(X; 2\pi R \ZZ) \sub \mr H^p_{\mr{dR}}(X)$. We define a filtration $F_k\Omega^p_{\cl, \ZZ}(X)$ on the space of closed $p$-forms with integral periods by the intersection $F_k\Omega^p(X) \cap \Omega^p_{\cl, \ZZ}(X)$.

\begin{remark}
Let $X$ be compact.  Then the space $\Omega^p_{\cl, \ZZ}(X) \sub \Omega^p(X)$ splits as $\d\Omega^{p-1}(X) \times \mr H^p(X; 2\pi R \ZZ)$, where the lattice factor is contained in the space of harmonic $p$-forms.  In particular, when we filter, the lattice factor is contained in the intersection of all the filtered pieces -- it does not vary with $k$.
\end{remark}

Recall from Lemma \ref{closed_Maxwell_relation_lemma} that there is a map on degree zero observables $F^* \colon \obssm_{\Omega^p_{cl},R}(U)_0 \to \obsq_{\pmp,R}(U)_0$ for \emph{any} $U$, which sends a compactly supported closed $p$-form $a$ to the local observable 
\[A \mapsto \int_U F_A \wedge \ast a.\]
However, this map is generally not an isomorphism.  Nevertheless, from a functional integral point-of-view we can define the expectation value of such an observable in the generalized Maxwell theory, even if $X$ has non-vanishing degree $p$ cohomology.  

\begin{definition} \label{exp_val_closed_def}
Given $\OO \in \obssm_{\Omega^p_{cl},R}(U)_0$ we extend $\OO$ to a global degree zero obervable using the factorization structure, and define its \emph{expectation value} to be
\[\langle \OO \rangle_R = \lim_{k \to \infty} \int_{F^k\Omega^p_{cl,\ZZ}(X)} \OO(a) e^{-S_R(a)} \d a.\]
\end{definition}

\begin{remark}
Note that expectation values are defined by first extending an observable using the factorization structure from the open set $U$ to the (compact oriented) global spacetime $X$.  As such, we only need to use the filtration of differential forms on compact $X$, where the spectrum of the Laplacian is well-behaved.
\end{remark}
 
We notice that if $\mr H^p(U) = 0$ then this definition agrees with the one we used in Section \ref{expvals_gaussian}, so in particular the limit converges.  In general, the integrand is dominated in absolute value by the integrand over all closed $p$-forms, which we already know converges (since the proof of \ref{expvals_gaussian} still applies with $\OO_i$ replaced by $|\OO_i|$).

\begin{remark}
The map $F^* \colon \obssm_{\Omega^p_{cl},R}(U)_0 \to \obsq_{\pmp,R}(U)_0$ is injective, and if $\mr H^p(U; \ZZ) = 0$ it is an isomorphism.  We can use Definition \ref{exp_val_closed_def} to define the expectation value of an observable in the image of $F^*$.  We do not know the right prescription for extending this notion to define the expectation value of observables that depend more generally on the discrete factor of the space of generalized Maxwell fields.
\end{remark}

Using this definition (and bearing in mind its relationship to the notion of expectation value considered above on certain open sets), we'll prove the main compatibility with duality.

\begin{theorem} \label{maintheorem}
Let $\OO$ be a local observable in $\obssm_{\Omega^p,R}(U)_0$, and let $\wt \OO\in \obssm_{\Omega^{n-p},1/2R}(U)_0$ be its Fourier dual observable.  Let $r(\OO)$ and $r(\wt \OO)$ be the restrictions to local observables in $\obssm_{\Omega^p_{cl}}(U)_0$ and $\obssm_{\Omega^{n-p}_{cl}}(U)_0$ respectively. Then, computing the expectation values of $r(\OO)$ and $r(\wt \OO)$, we find
\[\langle r(\OO) \rangle_R = \langle r(\wt \OO) \rangle_{\frac 1{2R}}.\]
\end{theorem}

\begin{remark}
Using the map discussed above we can just as well consider the observables $r(\OO)$ and $r(\wt \OO)$ as observables in the appropriate generalized Maxwell theories.
\end{remark}

\begin{proof}
 We know by \ref{doubledual} that $\OO = \wt {\wt \OO}$, so in particular $\langle r(\OO) \rangle_R = \langle r(\wt {\wt \OO}) \rangle_R$.  By the calculation in Proposition \ref{fourier_gaussian} we can write this expectation value as the limit as $k \to \infty$ of the Gaussian integrals
 \begin{align*}
  \frac 1{Z_k} \int_{F^k\Omega^p_{cl, \ZZ}(X)} \OO(a) e^{-S_R(a)} da &= \frac 1{Z_k} \int_{F^k\Omega^p_{cl, \ZZ}(X)} \wt{\wt \OO}(a) e^{-S_R(a)} da \\
  &= \frac 1{Z_k}\int_{F^k \Omega^p_{cl,\ZZ}(X)} \int_{F^k\Omega^{n-p}(X)} \wt\OO(\wt a) e^{-S_{1/2R}(\wt a) - i \int_{X} \wt a 
  \wedge a} d{\wt a} \,da \\
  &= \frac 1{Z_k}\int_{F^k \Omega^{p}(X)} \int_{F^k\Omega^{n-p}(X)} \wt \OO(\wt a) e^{-S_{1/2R}(\wt a) - i \int_{X} \wt a \wedge a} \, \delta_{\Omega^p_{cl,\ZZ}(X)}(a)\, d{\wt a} \, da \\
 \end{align*}
 The last line needs a little explanation.  The distribution $\delta_{\Omega^p_{cl,\ZZ}(X)}$ is the delta-function on the closed and integral $p$-forms sitting inside all $p$-forms (restricted to the filtered piece): pairing with this distribution and integrating over all $p$-forms in the filtered piece is the same as integrating only over the relevant subgroup. 
 
 Now, for a fixed value of $k$, we can reinterpret the final integral above by changing the order of integration.  This computes the Fourier dual of the delta function $\delta_{\Omega^p_{cl,\ZZ}(X)}$ and then pushes forward along the Hodge star. The Fourier dual of the delta function is $\delta_{\Omega^p_{\mr{cocl},\ZZ}(X)}$, the delta function on the group of \emph{coclosed} $p$-forms with integral $\d^*$ cohomology class.  That is, the external product $\delta_{\d^*\Omega^{p+1}(X)} \boxtimes \delta_{\mc \mr H^{p}_\ZZ}$ where $\mc \mr H^{p}_\ZZ$ is the lattice in the space of harmonic $p$-forms corresponding to the integral cohomology via Hodge theory.  Pushing this distribution forward along the Hodge star yields the delta function $\delta_{\Omega^{n-p}_{cl,\ZZ}(X)}$ on the closed $(n-p)$-forms with integral periods.  Therefore
\begin{align*}
 \langle r(\OO) \rangle_R &= \lim_{k \to \infty} \frac 1{Z_k}\int_{F^k \Omega^{n-p}(X)} \wt \OO(\wt a) e^{-S_{1/2R}(\wt a)} \delta_{\Omega^{n-p}_{cl,\ZZ}(X)} d\wt a \\
 &= \lim_{k \to \infty} \frac 1{Z_k}\int_{F^k \Omega^{n-p}_{cl, \ZZ}(X)} \wt \OO(\wt a) e^{-S_{1/2R}(\wt a)} d\wt a \\
 &= \langle r(\wt \OO) \rangle_{\frac 1{2R}}
\end{align*}
as required. 
\end{proof}

So to summarize, duality gives the following structure to the factorization algebra of quantum observables in our theories.
\begin{itemize}
 \item For each open set $U$, we have a subalgebra $\obssm_{\Omega^p_{cl},R}(U)_0 \sub \obsq_{R}(U)_0$ of the space of degree 0 local observables.  If $U$ is contractible (for instance for local observables in a small neighbourhood of a point) this subalgebra is dense.
 \item For a local observable $\OO$ living in this subalgebra we can define a \emph{Fourier dual} observable in $\obssm_{\Omega^{n-p}_{cl},1/2R}(U)_0$.  This depends on a choice of extension of $\OO$ to a functional on all $p$-forms, rather than just closed $p$-forms.
 \item For any choice of dual observable, we can compute their expectation values in the original theory and its dual, and they agree.  If $\mr H^p(U) = 0$ then this expectation value map agrees with a natural construction from the point of view of the factorization algebra.
\end{itemize}

We can rephrase the theorem in the language of factorization algebras.  Note that $\obssm_{\Omega^p}(U)_0$ and $\obssm_{\Omega^p_{cl}}(U)_0$ form factorization algebras themselves as $U$ varies, concentrated in degree zero.  The inclusion maps $\obssm_{\Omega^p}(U)_0 \to \obssm_{\Omega^p}(U)$ and $\obssm_{\Omega^p_{cl}}(U)_0 \to \obssm_{\Omega^p_{cl}}(U)$ are cochain maps since the target complexes are concentrated in non-positive degrees, and factorization algebra maps because the factorization algebra structure maps preserve the degree zero piece.  Likewise, the restriction maps  $\obssm_{\Omega^p}(U)_0 \to \obssm_{\Omega^p_{cl}}(U)_0$ clearly commute with the factorization algebra structure maps, so define a factorization algebra map.  

It's also easy to observe that the Fourier duality map $\obssm_{\Omega^p,R}(U)_0 \to \obssm_{\Omega^{n-p},\frac 1{2R}}(U)_0$ defines a factorization algebra map: the degree zero observables are just the free cdga on the local sections of a cosheaf of vector spaces, so the structure maps are just the maps induced on the free algebra from the cosheaf structure maps.  The structure maps are therefore given by the ordinary product
\begin{align*}
\sym(\Omega^p_c(U_1) / \d^*\Omega^{p+1}_c(U_1)) \otimes \sym(\Omega^p_c(U_2) / \d^*\Omega^{p+1}_c(U_2)) &\to \sym(\Omega^p_c(V) / \d^*\Omega^{p+1}_c(V)) \\
(\OO_1 \cdots \OO_n) \otimes (\OO'_1 \cdots \OO'_m) &\mapsto \OO_1 \cdots \OO_n \cdot \OO'_1 \cdots \OO'_m
\end{align*}
where $\OO_i$ and $\OO'_j$ are local linear smeared observables: sections of $\Omega^p_c / \d^*\Omega^{p+1}_c$ with compact support in disjoint open sets $U_1$ and $U_2$ respectively.  The Fourier transform of a product of observables with disjoint support is the product of the Fourier transforms, since if $\OO$ and $\OO'$ have disjoint support then their $L^2$ inner product is zero, so all Feynman diagrams with propagator edges between their vertices contribute zero weight. Therefore duality gives a factorization algebra map on the factorization algebra consisting of degree zero observables only.

Combining all of these statements we have a correspondence of (pre)factorization algebras of the form
\[\xymatrix{
 &&(\obssm_{\Omega^p,R})_0 \ar@{<->}^\sim[r] \ar[dl]_{r_R} &(\obssm_{\Omega^{n-p},\frac 1{2R}})_0 \ar[dr]^{r_{\frac 1{2R}}} \\
 \obsq_{p,R} &\obssm_{\Omega^p_{cl},R} \ar[l]_{F_R} &&& \obssm_{\Omega^{n-p}_{cl},\frac 1{2R}} \ar[r]^{F_{\frac 1{2R}}} &\obsq_{n-p,\frac 1{2R}}
}\]
where the top arrow is the isomorphism given by Fourier duality, and the diagonal arrows are given by restriction, then inclusion of degree zero observables into all observables.  

\begin{definition}
We say a pair of local observables $\OO$, $\OO'$ in $ \obsq_{p,R}(U)$ and $\obsq_{n-p,\frac 1{2R}}(U)$ respectively are \emph{incident} if they are the images under the restriction maps of Fourier dual degree zero observables.
\end{definition}

In this language, Theorem \ref{maintheorem} can be rephrased in the following way.
\begin{corollary}
If $\OO$ and $\OO'$ are incident local observables in dual generalized Maxwell theories then $\langle \OO \rangle_R = \langle \OO' \rangle_{\frac 1{2R}}$.
\end{corollary}

At the end of the next section we'll observe that it is not possible to improve this statement from a correspondence to a \emph{map} of prefactorization algebras (even of the factorization algebras in the closed $p$-form theories) in any natural way.

\begin{remark}
We should emphasise that, while our results allow us to describe pairs of dual observables, and shows that such pairs have equal expectation value, they do not allow us to actually \emph{calculate} the values of these expectation values.  This is in contrast to much of the existing work on abelian duality, as described in the introduction, which has studied duality for more specific observables in specific dimensions (particularly the trivial observable $1$ whose expectation value is the partition function of the theory), but in a formalism in which the common values of these expectations can be calculated and directly shown to be equal.
\end{remark}

\subsection{Wilson and `t Hooft Operators} \label{sectionWilsontHooft}
In this section we'll give a concrete example of observables that admit canonical duals in generalized Maxwell theories, corresponding to familiar observables in the usual Maxwell theory, i.e. the case $p=2$.  For Wilson and 't Hooft operators in dimensions 3 and 4 specifically, the behaviour under abelian duality is discussed in a paper of Kapustin and Tikhonov \cite{KapustinTikhonov}.

Wilson and 't Hooft operators can be defined classically in Yang-Mills theory with any compact gauge group $G$, just as functionals on the space of fields (from which we will -- for an abelian gauge group -- construct classical and quantum observables).  So let $X$, for the moment, be a Riemannian 4-manifold.
\begin{definition}
Let $\rho$ be an irreducible representation of $G$.  The \emph{Wilson operator} $W_{\gamma,\rho}$ around an oriented loop $\gamma$ in $X$ is the functional on the space of connections on principal $G$-bundles sending a connection $A$ to
\[W_{\gamma,\rho}(A) = \tr(\rho(\hol_\gamma(A))),\]
where $\hol_\gamma(A)$ denotes the holonomy of the connection around the loop $\gamma$.  Equivalently we can compute the Wilson operator as a path-ordered exponential
\[W_{\gamma,\rho}(A) = \tr(\rho(\mc P e^{i \oint_\gamma A})).\]
\end{definition}
Suppose $\gamma$ bounds a disc $D$.  In this case there is a candidate dual observable to the Wilson operator.
\begin{definition}
Let $\mu \colon U(1) \to G$ be a cocharacter for the group $G$.  The \emph{'t Hooft operator} $T_{\gamma, \mu}$ around the loop $\gamma$ in $X$ is the functional on the space of connections on principal $G$-bundles sending a connection $A$ to 
\[T_{\gamma, \mu}(A) = e^{i \int_D \mu^* (\ast F_A)}\]
where $F_A$ is the curvature of $A$, $\ast F_A$ is its Hodge star, and where $\mu^* \colon \Omega^2(X; \gg^*) \to \Omega^2(X)$ is the pullback along the cocharacter.
\end{definition}
The relationship between these two kinds of operator is clearest in the abelian case, so let $G = U(1)$ (or, with minor modifications, any torus).  The irreducible representations of $U(1)$ are given by the $n$-power maps $z \mapsto z^n$ for $n \in \ZZ$, so we can write our Wilson operators as
\[W_{\gamma, n}(A) = e^{i n \oint_\gamma A} = e^{in \int_D F_A}\]
assuming $\partial D = \gamma$ as above.  Cocharacters are also indexed by integers, so similarly we can describe the 't Hooft operators as
\[T_{\gamma, m}(A) = e^{i m \int_D \ast F_A}.\]
As described above we compute the dual of an observable in abelian Yang-Mills theory by taking its Fourier dual as a functional on all 2-forms, then precomposing with the Hodge star.  The Fourier dual of a plane wave is a plane wave, so we should expect Wilson and 't Hooft observables to be dual to one another.  In the rest of this section we'll prove this, and generalize it to higher degree theories.

Consider the degree $p$ generalized Maxwell theory with coupling constant $R$ on an $n$-manifold $X$.  We'll first describe degree zero gauge invariant observables associated to a complex number $r$ and a singular chain $C \in C_{p}(U)$, for $U \sub X$ an open set.   Recall that the local degree zero observables are given by $\OO(\mr H^0(\Phi(U))) \sub \OO(T^*[-1]\Phi(U))$, and that in this case the degree zero cohomology $\mr H^0(\Phi)$ is given by the group $\Omega^{p}_{cl, \ZZ}(U; \CC)$ of closed $p$-forms with integral periods (periods in the lattice $\mr H^p(U; 2\pi R\ZZ)$).  So, analogously to the above we define a \emph{Wilson}-type operator by
\[W_{C,r}(\alpha) = e^{i r \int_C \alpha}.\]
Similarly, if $C$ is instead a chain in $C_{n-p}(U)$, we can define an \emph{'t Hooft}-type operator by first applying the Hodge star:
\[T_{C,r}(\alpha) = e^{i r \int_C \ast \alpha}.\]

\begin{remark}
These operators don't quite arise from our definitions: they aren't polynomial functions in linear observables.  However, they can be arbitrarily well approximated by polynomials by taking a finite number of terms in the Taylor series.  We should either note that our constructions, in which the observables are described by a symmetric algebra, extend to completed symmetric algebras, or equivalently just interpret claims about duality for these observables as claims about these polynomial approximations at every degree.
\end{remark}

\begin{remark}
These operators are particularly interesting in the case where they don't quite bound a disk: when they represent \emph{torsion} classes of the homology of $X$.  Such classes are crucial in the analysis of Freed, Moore and Segal \cite{FMS1,FMS2} where non-trivial commutation relations are established between corresponding flux states in the Hilbert space.  More recently Becker, Benini, Schenkel and Szabo \cite{BBSS} analysed this non-commutativity, and abelian duality more generally, in the Lorentzian setting using the language of AQFT.  It would be interesting to understand how abelian duality as described in this article applies to these torsion operators, but the results of this work, to our knowledge, will not contribute a novel perspective to their most salient property: their commutation relations.
\end{remark}

Now, let's investigate duality for these observables.  Firstly, suppose $U$ is an open set with $\mr H^p(U) = 0$, so that the condition of having integral periods is trivial.  Then the observables defined above immediately lift to observables in the closed $p$-form theory, and admit canonical extensions to observables in the theory where fields are \emph{all} $p$-forms (given by precisely the same formula).  We can also investigate approximations for these observables by \emph{smeared} observables.  Integration over a $p$-chain $C$ can be written as the $L^2$-pairing with a particular current: the delta function $\delta_C$.  This current can, in turn, be approximated in $L^2$ by $p$-forms supported on small neighbourhoods of $C$.  So let's investigate the Fourier dual of the smeared observable 
\[\OO_\beta(a) = e^{i r \int_X a \wedge \ast \beta}\]
where $\beta$ is a $p$-form.  

We'll compute this dual using the functional integral at a regularized level:
\begin{align*}
\wt {\OO_{\beta,r}} (\wt a) &= \lim_{k \to \infty} \frac 1{Z_k} \left(\int_{F^k\Omega^p(X)} e^{-S_R(a) + i \int_{X} \wt a \wedge a + ir \int_X a \wedge \ast \beta} da\right) e^{S_{1/2R}(\wt a)} \\
&= e^{-\frac{r^2}{4R^2}\|\beta\|^2} e^{\frac{-r}{2R^2} \int_X \wt a \wedge \beta} \\
&= e^{-\frac{r^2}{4R^2}\|\beta\|^2} \cdot \OO_{\ast \beta, ir/2R^2}(\wt a).
\end{align*}
This calculation allows us to produce the dual of the original Wilson operator by dualising increasingly good smooth approximations.  We find
\[\wt {W_{C,r}} = e^{-\frac{r^2}{4R^2}\|C\|^2}T_{C, ir/2R^2}\]
where $\|C\|$ is the $L^2$-norm of the chain $C$: the usual $L^2$ norm with respect to the metric of its image under Poincar\'e duality.

To summarise, duality for (generalized) Wilson and 't Hooft operators tells us the following.
\begin{corollary}
There is an equality of expectation values in generalized Maxwell theories
\[ \langle W_{C,r} \rangle_R = e^{-\frac{r^2}{4R^2}\|C\|^2} \langle T_{C, ir/2R^2} \rangle_{\frac 1{2R}}.\]
\end{corollary}

\begin{example}
Kapustin and Tikhonov \cite{KapustinTikhonov} studied duality for Wilson and 't Hooft type operators in 3 and 4 dimensions (in addition to duality for boundary conditions, which we have not addressed in the present work) by considering composition of operators with duality domain walls.  Their analysis includes several instances of the duality described above.  For 4-dimensional abelian gauge theories they discuss the usual duality of Wilson and 't Hooft line operators, and in 3-dimensions they discuss the duality between order and disorder line (and local) operators in 3d abelian gauge theories and 3d circle valued sigma models.

Kapustin and Tikhonov discuss duality for one further type of non-local operator -- a Chern-Simons operator supported on a 3-manifold. Our  Our formalism does not allow us to analyze this operator, since it does not factor through the closed 2-form theory (i.e. it does not depend only on the curvature of the gauge field).
\end{example}

\begin{remark}
We described a canonical dual for Wilson and 't Hooft operators, using a natural choice of lift from operators acting on closed $p$-forms to operators acting on all $p$-forms.  It's natural to ask whether it's possible to do this for all observables, thus promoting abelian duality from a correspondence to a genuine map of prefactorization algebras.  It turns out however that this is impossible.  We'll demonstrate this in a specific example.

Let $X$ be a $2p$-manifold satisfying $\mr H^p(X) = 0$, so $\Omega^p(X)$ splits as $\d\Omega^{p-1}(X) \oplus \d^*\Omega^{p+1}(X)$.  We'll discuss global linear smeared observables in the closed $p$-form theory on $X$.  Such an observable is an element of $\d\Omega^{p-1}(X))^\vee$ given by $L^2$-pairing with an exact $p$-form.  Prescribing an extension of an observable in the closed $p$-form theory to an observable in the \emph{full} $p$-form theory is equivalent to prescribing the action of the lifted observable on coexact $p$-forms, so a choice of such an extension for all linear smeared observables is a map
\[f \colon \d\Omega^{p-1}(X) \to \d^*\Omega^{p+1}(X).\]
Having specified such a map we obtain a canonical dual observable for every linear observable in the closed $p$-form theory.  Applying this duality procedure twice should bring us back to the observable we started with, i.e. $\ast f \ast f = \id$, which means in particular that $f$ must be an isomorphism.

However, we also need compatibility with the canonical duals constructed above for Wilson and 't Hooft operators.  Let $I$ be a linear observable of form ``integrate over a $p$-cycle'', and let $\alpha_i \to I$ be a sequence of smeared observables approximating $I$.  The duals of Wilson observables of form $e^I$ are determined by the duals of linear observables since $e^x$ is Fourier self-dual, and to obtain the required dual for our Wilson operator we need to choose the trivial lift for $I$, i.e. we need $f(\alpha_i) \to 0$.  But then $\ast f \ast f(\alpha_i) \to 0$, so $\alpha_i \to 0$ which is false.  So there is no possible canonical lift $f$ compatible with the natural duals for Wilson and 't Hooft operators, and in particular no way of improving abelian duality to a genuine map rather than just a correspondence.
\end{remark}

\pagestyle{bib}
\bibliographystyle{alpha}
\bibliography{Abelian_duality_2}

\end{document}